\documentclass{SCIYA2017enOL}
\online
\begin{document}

\ensubject{fdsfd}

\ArticleType{ARTICLES}
\Year{2017}
\Month{January}%
\Vol{60}
\No{1}
\BeginPage{1} %
\DOI{10.1007/s11425-000-0000-0}
\ReceiveDate{January 1, 2017}
\AcceptDate{January 1, 2017}

\title[]{Characterization of Image Spaces of Riemann-Liouville Fractional Integral Operators on Sobolev Spaces $W^{m,p}(\Omega)$}
{Characterization of Image Spaces of Riemann-Liouville Fractional Integral Operators on Sobolev Spaces $W^{m,p}(\Omega)$}

\author[1,$\ast$]{Lijing Zhao}{{zhaolj17@nwpu.edu.cn}}
\author[2]{Weihua Deng}{dengwh@lzu.edu.cn}
\author[3]{Jan S. Hesthaven}{Jan.Hesthaven@epfl.ch}



\address[1]{School of Mathematics and Statistics, \\
Northwestern Polytechnical University, Xi'an {\rm 710129}, China}
\address[2]{School of Mathematics and Statistics, Gansu Key Laboratory of Applied Mathematics and Complex Systems, \\
Lanzhou University, Lanzhou {\rm 730000}, China}
\address[3]{EPFL-SB-MATHICSE-MCSS, \'{E}cole Polyt\'{e}chnique F\'{e}d\'{e}rale de Lausanne, Lausanne {\rm CH-1015}, Switzerland}

\abstract{Fractional operators are widely used in mathematical models describing abnormal and nonlocal phenomena. Although there are extensive numerical methods for solving the corresponding model problems, theoretical analysis such as the regularity result, or the relationship between the left-side and right-side fractional operators are seldom mentioned. In stead of considering the fractional derivative spaces, this paper starts from discussing the image spaces of Riemann-Liouville fractional integrals of $L_p(\Omega)$ functions, since the fractional derivative operators that often used are all pseudo-differential. Then high regularity situation---the image spaces of Riemann-Liouville fractional integral operators on $W^{m,p}(\Omega)$ space are considered. Equivalent characterizations of the defined spaces, as well as of the intersection of the left-side and right-side spaces are given. The behavior of the functions in the defined spaces at both the nearby boundary point/ponits and the points in the domain are demonstrated in a clear way. Besides, tempered fractional operators show to be reciprocal to the corresponding Riemann-Liouville fractional operators, which is expected to make some efforts on theoretical support for relevant numerical methods. Last, we also provide some instructions on how to take advantage of the introduced spaces when numerically solving fractional equations.}



\maketitle

\section{Introduction}\label{section:1}

The concept of fractional derivatives are almost as old as their more familiar integer order counterparts. The notion of fractional calculus goes back to the Leibniz's note in his list to L'Hospital
dated Sep. 30, 1695~\cite{Butzer:00,Kenneth:93,Podlubny:99}.
Until recently, however, fractional derivatives have been widely and successfully explored as a tool for developing more sophisticated mathematical models. During the last decades, fractional calculus
has gained widely attention in the development of the
mathematical models in statistical physics, mathematical finance,
polymer modeling; fractional dynamics are also
disclosed, such as, anomalous diffusion, synchronization of chaos,
and multi-directional multi-scroll attractors
~\cite{Bouchaud:90,Metzler:00}.

At the same time, there are numerous numerical methods have been proposed to solve temporal as well as spatial fractional problems, such as finite difference methods~\cite{Meerschaert:04,Deng:07d,Tian:12,Zhao:14,Alikhanov:15,Zhao:16,Jin:17,Zeng:18,Deng:19}, finite element methods~\cite{Ervin:06,Deng:08,Ervin:16}, spectral methods~\cite{Li:09,Tian:14,Wang:15,Zayernouri:15,Chen:16,Jiao:16,Mao:18}, Monte Carlo approach~\cite{Magdziarz:07,Kyprianou:18}, and so on.

Generally speaking, high order method requires high regularity of the solution, most of time, the numerical methods are designed under the assumption that the solution of the fractional problem is smooth. However, the regularity results about the fractional operators are seldom mentioned in a clear way. Mostly, the solution is assumed in a space, say fractional derivative space, which are defined as the closure of $C_0^\infty(\Omega)$ under the norm of fractional Sobolev spaces on $\mathbb{R}$. However, since fractional operators
are not only pseudo-differential, but also non-local, the singularity/singularites at the original or boundary points indeed exist. Whether or how this influences or pollutes the whole domain has seldom been spoke of. In general,
there are four questions lie in:

1. Under which framework the ``smooth" should be defined for fractional operators.

2. What is the regularity property of the fractional operators under the above framework.

3. How to describe clearly the singularity/singularites behavior near the original or boundary points.

4. How do/does the singularity/singularites affect the whole domain.

5. What is the relationship between the left-side and right-side fractional operators.

Actually, the theoretical results about fractional operators have been studied in the early days~\cite{Miller:93,Oldham:74,Podlubny:99,Samko:93}. In this paper, we collect (not simply copy but sometimes have to
flip through pages, relabeling and rearrangement the content) some concepts that introduced in~\cite{Samko:93}, especially the image space of Riemann-Liouville fractional integrals of $L_p$ functions, and give some elaboration on it.
The reason we choose this typical space, not only because it comes from a ``non state of the art" references, but also because the key difficulty of the common used fractional derivative operators, defined such as in Riemann-Liouville or Caputo way, is actually come from the pseudo-differential or fractional integral operator in them. While the image space of fractional integrals can catch this characteristic very well. So it is a natural way to begin our discussion from it. After extending the space introduced in~\cite{Samko:93} to the image spaces of Riemann-Liouville fractional integral operators on $W^{m,p}(\Omega)$ functions, this paper targets on answering the above listed five questions, and also provide some instructions on how to apply these spaces into numerically solving fractional equations.

The rest of this paper is organized as follows. In Section \ref{section:2}, we collect the properties of the image spaces of Riemann-Liouville fractional integrals of $L_p(\Omega)$ functions, then make a further discussion about the spaces, and show some very useful results, including the equivalent representation of the intersection of left-side and right-side spaces.
In Section \ref{section:3}, we extend the space to high regularity case, and discuss the image spaces of fractional integrals on Sobolev space $W^{m,p}(\Omega)$. We can see that the functions in the defined spaces consist of two parts: one is smooth in the domain but possesses specific regularity at the boundary point; another one shows the regularity in the domain and the asymptotic behavior tending to zero near the boundary point. The regularity results and approximation properties are also demonstrated. In Section \ref{section:4}, we show a brief instruction on how to make used of the introduced spaces on solving fractional problems.
Finally, the main results are summarized in Section \ref{section:5}.

\section{Image Space of Riemann-Liouville Integrals on $L_p(\Omega)$ Space}\label{section:2}

The concept of fractional integrals of functions in $L_p(\Omega)$ have been proposed and studied by Samko et al. in~\cite{Samko:93}, the set of which actually can be normed as a space, which we denote in this paper as $I_{a+}^{\alpha}[L_p(\Omega)]$ or $I_{b-}^{\alpha}[L_p(\Omega)]$. In this section, we collect some properties of this kind of spaces, and give a further discussion about them. Based on these, in next section, the image spaces are extended to  high regularity cases, i.e., $I_{a+}^{\alpha}[W^{m,p}(\Omega)]$ and $I_{b-}^{\alpha}[W^{m,p}(\Omega)]$, as well as their interaction $I_{a+}^{\alpha}[W^{m,p}(\Omega)]\bigcap I_{b-}^{\alpha}[W^{m,p}(\Omega)]$, which are fundamentally different from the existing spaces.
\subsection{Definitions}\label{subsection:2.1}

We denote by $L_p(\Omega)~ (1\le p<\infty)$ the set of Lebesgue measurable functions
$f$ on $\Omega=[a,b]$ (a \textbf{closed} interval) for which $\|f\|_{L_p}<\infty$, where $\|f\|_{p}=\left(\int_{\Omega} |f(x)|^p dx\right)^{1/p}$, $1\le p<\infty$.

Let $AC(\Omega)$ be the space of functions $f$ which are absolutely continuous on $\Omega$.
It is known that $AC(\Omega)$ coincides with the space of primitives of
Lebesgue summable functions (\cite{Kolmogorov:68}, p. 338):
\begin{equation}\label{equation:2.1.1}
f(x)\in AC(\Omega)\Leftrightarrow f(x)=c+\int_{a}^x \phi(x)dx,~~~\phi(x)\in L_1(\Omega),
\end{equation}
where $\phi(x)=f'(x)$, and $c=f(a)$.

Denote by $AC^n(\Omega)$ the space of functions $f(x)$ which have continuous derivatives up to
order $(n-1)$ on $\Omega$ such that $f^{(n-1)}(x)\in AC(\Omega)$:
\begin{equation}\label{equation:2.1.2}
AC^n(\Omega)=\{f:D^{n-1}f(x)\in AC(\Omega)\}.
\end{equation}
In particular, $AC^1(\Omega)=AC(\Omega)$.

Let $u(x)\in L_1(\Omega)$. Denote $_{a}I_{x}^{\beta}$ and $_{x}I_b^{\beta}$ separately as the $\alpha$-th order left-side Riemann-Liouville fractional integral operator
$$_{a}I_{x}^{\beta}u(x):=\frac{1}{\Gamma(\beta)}\int_{a}^x(x-s)^{\beta-1}u(s)ds,$$
and right-side Riemann-Liouville fractional integral operator
$$_{x}I_b^{\beta}u(x):=\frac{1}{\Gamma(\beta)}\int_{x}^b(s-x)^{\beta-1}u(s)ds.$$
Here $\Gamma(\cdot)$ presents the Euler gamma function.

Now we can introduce the set of $\alpha$-th order left-side and right-side Riemann-Liouville fractional integrals of $L_p(\Omega)$ functions, $1\le p<\infty$.
\begin{definition}\label{defi_frac_int_space}
Let $\alpha>0$. We call
\begin{equation}\label{defi_left}
I_{a+}^\alpha[L_p(\Omega)]:=\left\{f:f(x)={}_{a}I_{x}^\alpha \varphi(x), \varphi(x)\in L_p(\Omega), x\in \Omega\right\},
\end{equation}
and
\begin{equation}\label{defi_right}
I_{b-}^{\alpha}[L_p(\Omega)]:=\left\{f:f(x)={}_{x}I_{b}^\alpha \varphi(x), \varphi(x)\in L_p(\Omega), x\in \Omega\right\},
\end{equation}
as the sets of $\alpha$-th order left-side and right-side Riemann-Liouville integrals of $L_p(\Omega)$ functions, respectively.
\end{definition}

Now we only discuss $I_{a+}^\alpha[L_p(\Omega)]$; similar results can be derived for $I_{b-}^{\alpha}[L_p(\Omega)]$.

The following lemma shows that the sets defined above indeed make sense.

\begin{lemma}\label{lemma:2.1.1}
If $u(x)\in I_{a+}^\alpha[L_p(\Omega)]$, then there exists a unique $\varphi(x)\in L_p(\Omega)$, such that
$u(x)={}_{a}I_{x}^\alpha \varphi(x)$.
\end{lemma}
\begin{proof}
If there are two functions $\varphi_1(x)$ and $\varphi_2(x)$ in $L_p(\Omega)$, such that
$$u(x)={}_{a}I_{x}^\alpha \varphi_1(x)={}_{a}I_{x}^\alpha \varphi_2(x),$$
Suppose $n-1\leq \alpha<n$. Denote $\alpha'=\alpha-(n-1)\in[0,1)$.
Then $$D^{n-1}{}_{a}I_{x}^\alpha[\varphi_1(x)-\varphi_2(x)]
={}_{a}I_{x}^{\alpha'}[\varphi_1(x)-\varphi_2(x)]=0,$$
i.e.,
$$\frac{1}{\Gamma(\alpha')}\int_{a}^{x}(x-s)^{\alpha'-1} (\varphi_1(s)-\varphi_2(s))\,ds= 0.$$
This is actually an Abel's equation~\cite{Gakhov:66}. By the exsitance and uniqueness of the solution, we have
$$\varphi_1(x)-\varphi_2(x)=0 \textrm{~a.e.~on~}\Omega=[a,b],$$
i.e., $\varphi_1(x)=\varphi_2(x)$ in $L_p(\Omega)$.
\end{proof}

Since for $p\geq 1$,
$I_{a+}^\alpha[L_p(\Omega)]\hookrightarrow L_p(\Omega)$(Theorem 2.6 in\cite{Samko:93}), by Lemma \ref{lemma:2.1.1}, $I_{a+}^\alpha:L_p(\Omega)\rightarrow L_p(\Omega)$ is an injection. So
we can introduce the norm in $I_{a+}^\alpha[L_p(\Omega)]$ by
\begin{equation}\label{norm}
\|u(x)\|_{I_{a+}^\alpha[L_p(\Omega)]}:=\|\varphi(x)\|_{L_p(\Omega)},
\end{equation}
where $u(x)={}_{a}I_{x}^\alpha \varphi(x)$.

Now we can see that $I_{a+}^\alpha[L_p(\Omega)]$ with the norm
(\ref{norm}) is a Banach space as it is isometric to $L_p(\Omega)$.

\subsection{Properties}\label{subsection:2.2}

\subsubsection{Characteristics}\label{subsection:2.2.1}

One of the characteristics of the space $I_{a+}^\alpha[L_1(\Omega)]$ is given by the following lemma.
\begin{lemma}[\rm~\cite{Samko:93}, Theorem 2.3]\label{lemma:2.2.1.1}
In order that $u(x)\in I_{a+}^\alpha[L_1(\Omega)]$, $n-1\le \alpha <n$, it is necessary and
sufficient that
$${}_{a}I_{x}^{n-\alpha}u(x)\in AC^n(\Omega),$$
and
\begin{equation}\label{equation:2.2.4}
{}^{RL}_{a}D_{x}^{\alpha-k}u(x)\big|_{x=a}:=D^{n-k}{}_{a}I_{x}^{n-\alpha}u(x)\big|_{x=a}=0,~~k=1,2,\cdots,n.
\end{equation}
In particular, when $0<\alpha<1$, then $u(x)\in I_{a+}^\alpha[L_1(\Omega)]$ is equivalent to
$${}_{a}I_{x}^{1-\alpha}u(x)\in AC(\Omega),$$
and
$${}_{a}I_{x}^{1-\alpha}u(x)\big|_{x=a}=0.$$
\end{lemma}

For simplicity, in the rest of this paper, we replace ${}^{RL}_{a}D_{x}^{\alpha}u(x):=D^{n}{}_{a}I_{x}^{n-\alpha}u(x)$ with ${}_{a}D_{x}^{\alpha}$, and replace ${}^{RL}_{x}D_{b}^{\alpha}u(x):=(-1)^n\,D^{n}{}_{x}I_{b}^{n-\alpha}u(x)$ with ${}_{x}D_{b}^{\alpha}$, for short.

\begin{remark}\label{remark:3}
From Lemma \ref{lemma:2.2.1.1}, we can see that
\begin{enumerate}
\renewcommand{\labelenumi}{(\roman{enumi})}
\item
when $0<\alpha<1$, $u(x)\in {}_{a}I_{x}^{\alpha}[L_1(\Omega)]$
\textbf{does not need to be zero} at the left boundary. In fact, $u(a)$ is permitted to be a constant
or even infinite. For instance, assume $0<\alpha\le\alpha'<1$,
and $u(x)=(x-a)^{\alpha-\alpha'}$, then
$$
{}_{a}I_{x}^{1-\alpha}u(x)=\frac{\Gamma(\alpha-\alpha'+1)}{\Gamma(2-\alpha')}(x-a)^{1-\alpha'}
\in AC(\Omega),
$$
and $${}_{a}I_{x}^{1-\alpha}u(x)\big|_{x=a}=0.$$
Thus, by Lemma \ref{lemma:2.2.1.1}, $u(x)\in I_{a+}^{\alpha}[L_1(\Omega)]$.
While
$$u(x)\neq0 \textrm{~as~} x\rightarrow a.$$

\item
if $n-1\leq \alpha<n$, then for $k=1,2,\cdots,n$,
$$(x-a)^{\alpha-k}\notin I_{a+}^\alpha[L_1(\Omega)],$$
since by~\cite{Podlubny:99},
$${}_{a}D_{x}^{\alpha-k}(x-a)^{\alpha-k}\equiv\Gamma(\alpha-k+1)\neq 0.$$
Actually, from the definition of the operator ${}_{a}I_{x}^{\alpha}$, there is
$$(x-a)^{\alpha-1}=\Gamma(\alpha)\cdot{}_{a}I_{x}^{\alpha}\delta(x-a).$$
\item
for fixed $p\geq 1$, if $0<\alpha<\frac{1}{p}$, then $(x-a)^{-\alpha}\in L_p(\Omega)$, and
$$1=\frac{1}{\Gamma(1-\alpha)}\cdot{}_{a}I_{x}^{\alpha}(x-a)^{-\alpha},$$
thus,
$$1\in I_{a+}^\alpha[L_p(\Omega)],\quad 0<\alpha<\frac{1}{p}.$$
Thereby, given $p\geq 1$, and $\beta>-\frac{1}{p}$, then
$$(x-a)^{\beta}\in I_{a+}^\alpha[L_p(\Omega)],\quad 0<\alpha<\beta+\frac{1}{p}.$$
\end{enumerate}
\end{remark}

\begin{remark}\label{remark:1}
Here one fact need to be clarified: if $u(x)\in I_{a+}^\alpha[L_1(\Omega)]$, $n-1\le \alpha <n$,
then by the definitions of $AC(\Omega)$ as well as $AC^n(\Omega)$ in (\ref{equation:2.1.1})-(\ref{equation:2.1.2}), and Lemma \ref{lemma:2.2.1.1}, there is
${}_{a}D_{x}^{\alpha}u(x)=D^n {}_{a}I_{x}^{n-\alpha}u(x)\in L_1(\Omega).$

While on the other hand, ${}_{a}D_{x}^{\alpha}u(x)\in L_1(\Omega)$ does not mean ${}_{a}I_{x}^{n-\alpha}u(x)\in AC^n(\Omega)$ (since $f'(x)$ being summable \textbf{cannot ensure} that Newton-leibniz's formula $\int_{a}^{x}f'(x)\,dx=f(x)-f(a)$ holds). Consequently, by Lemma \ref{lemma:2.2.1.1}, $u(x)\in I_{a+}^\alpha[L_1(\Omega)]$ \textbf{might not be true}.

In other words,
$$u(x)\in I_{a+}^\alpha[L_1(\Omega)]\overset{\nLeftarrow}\Rightarrow {}_{a}D_{x}^{\alpha}u(x)\in L_1(\Omega).$$

For example, for $n-1\leq \alpha<n$, let $u(x)=(x-a)^{\alpha-k}$, where $k\leq n$ is a positive integer. Then,
$${}_{a}D_{x}^{\alpha}u(x)=D^n{}_{a}I_{x}^{n-\alpha}u(x)
=\frac{\Gamma(\alpha-k+1)}{\Gamma(n-k+1)}D^n[(x-a)^{n-k}]=0\in L_1(\Omega).$$
However, $u(x)\notin I_{a+}^\alpha[L_1(\Omega)]$, as we mentioned in Remark \ref{remark:3}. So, $I_{a+}^\alpha[L_1(\Omega)]$ is different from the exist fractional derivative spaces on $\Omega$ defined as the closure of $C_0^{\infty}$ under the norms of the fractional derivative spaces on $\mathbb{R}$.
\end{remark}

\begin{remark}\label{remark:4}
Recall the following formulas (\cite{Askey:75}, p.20):

$$\,_{-1}I_{x}^{\alpha}\big((1+x)^{\delta} J_{n}^{\gamma,\delta}(x)\big)
=\frac{\Gamma(n+\delta+1)}{\Gamma(n+\delta+\alpha+1)}
(1+x)^{\delta+\alpha} J_{n}^{\gamma-\alpha,\delta+\alpha}(x),$$
$$\,_{x}I_{1}^{\alpha}\big((1-x)^{\delta} J_{n}^{\delta,\gamma}(x)\big)
=\frac{\Gamma(n+\delta+1)}{\Gamma(n+\delta+\alpha+1)}
(1-x)^{\delta+\alpha} J_{n}^{\delta+\alpha,\gamma-\alpha}(x),$$
where $\alpha>0$, $\delta>-1$, $\gamma\in\mathbb{R}$, and $J_{n}^{\sigma,\eta}(x)$ denote the Jacobi polynomials. These functions are always used in spectral methods for solving fractional problems~\cite{Tian:14,Zayernouri:15}. In fact, since both $(1+x)^{\delta} J_{n}^{\gamma,\delta}(x)$ and $(1-x)^{\delta} J_{n}^{\delta,\gamma}(x)$
belong to $L_1([-1,1])$, we can see that
$$(1+x)^{\delta+\alpha} J_{n}^{\gamma-\alpha,\delta+\alpha}(x)\in I_{(-1)+}^{\alpha}[L_1([-1,1])],$$
and
$$(1-x)^{\delta+\alpha} J_{n}^{\delta+\alpha,\gamma-\alpha}(x)\in I_{1-}^{\alpha}[L_1([-1,1])].$$

\end{remark}

\subsubsection{Properties}\label{subsection:2.2.2}

The key property of the space $I_{a+}^\alpha[L_1(\Omega)]$ is listed in the following lemma.

\begin{lemma}\label{lemma:2.2.2.1}
If $u(x)\in I_{a+}^\alpha[L_1(\Omega)]$, then 
\begin{equation}\label{equation:2.2.5}
{}_{a}I_{x}^{\alpha}\,{}_{a}D_{x}^{\alpha}u(x)=u(x), \quad x\in \Omega.
\end{equation}
\end{lemma}
\begin{remark}
A parallel result can also see in Theorem 2.22 of~\cite{Diethelm:10}.
\end{remark}
\begin{proof}
Firstly, the equality 
\begin{equation}\label{equation:2.2.6}
{}_{a}D_{x}^{\alpha}{}_{a}I_{x}^{\alpha}\varphi(x)=\varphi(x), \quad x\in \Omega,
\end{equation}
is valid for any summable function $\varphi(x)$. This is because, since $\varphi(x)\in L_1(\Omega)$,
${}_{a}I_{x}^{\alpha}\varphi(x)\in I_{a+}^\alpha[L_1(\Omega)]$. Thus by Lemma \ref{lemma:2.2.1.1},
$${}_{a}I_{x}^{n-\alpha}\left[{}_{a}I_{x}^{\alpha}\varphi(x)\right]\in AC^{n}(\Omega).$$
Therefore, by the definitions of $AC(\Omega)$ as well as $AC^n(\Omega)$ in (\ref{equation:2.1.1})-(\ref{equation:2.1.2}),
$${}_{a}D_{x}^{\alpha}{}_{a}I_{x}^{\alpha}\varphi(x)
={}_{a}D_{x}^{n}{}_{a}I_{x}^{n-\alpha}\left[{}_{a}I_{x}^{\alpha}\varphi(x)\right]
={}_{a}D_{x}^{n}\left[{}_{a}I_{x}^{n}\varphi(x)\right]=\varphi(x).$$

Secondly, if $u(x)\in I_{a+}^\alpha[L_1(\Omega)]$, there exists $\varphi(x)\in L_1(\Omega)$, such that
$u(x)={}_{a}I_{x}^{\alpha}\varphi(x)$. Then by Eq. (\ref{equation:2.2.6}),
$${}_{a}I_{x}^{\alpha}{}_{a}D_{x}^{\alpha}u(x)
={}_{a}I_{x}^{\alpha}\left[{}_{a}D_{x}^{\alpha}{}_{a}I_{x}^{\alpha}\varphi(x)\right]
={}_{a}I_{x}^{\alpha}\varphi(x)=u(x).$$
\end{proof}

Now Lemma \ref{lemma:2.1.1} can be rewritten as:

If $u(x)\in I_{a+}^\alpha[L_p(\Omega)]$, $1\leq p<\infty$, then there exists a unique $\varphi(x)\in L_p(\Omega)$, such that
$u(x)={}_{a}I_{x}^\alpha \varphi(x)$, where $\varphi(x)={}_{a}D_{x}^\alpha u(x)$.

The norm in (\ref{norm}) can be redefined as:
\begin{equation}\label{norm_redefined}
\|u(x)\|_{I_{a+}^\alpha[L_p(\Omega)]}:=\|{}_{a}D_{x}^\alpha u(x)\|_{L_p(\Omega)}.
\end{equation}

\begin{remark}\label{remark:4}

If $u(x)\in I_{a+}^{\alpha}[L_1(\Omega)]$, then there exist a unique $\varphi(x)\in L_1(\Omega)$, such that $u(x)={}_{a}I_{x}^{\alpha} \varphi(x)$. Using Lemma \ref{lemma:2.2.2.1}, we have
\begin{equation}\label{weak_solution}
\begin{array}{lll}
&&\int_a^b u(x)\cdot {}_{x}D_{b}^{\alpha}\phi(x)\,dx\vspace{3pt}\\
&=&\int_a^b {}_{a}I_{x}^{\alpha} \varphi(x)\cdot {}_{x}D_{b}^{\alpha}\phi(x)\,dx\vspace{3pt}\\
&=&\int_a^b \varphi(x)\cdot {}_{x}I_{b}^{\alpha}{}{}_{x}D_{b}^{\alpha}\phi(x)\,dx\vspace{3pt}\\
&=&\int_a^b \varphi(x)\phi(x)\,dx \qquad \qquad \qquad\forall \phi(x)\in C^{\infty}_{c}(a,b),
\end{array}
\end{equation}
where the integrals make sense because of the H\"{o}lder inequality $\|fg\|_{L_1}\leq \|f\|_{L_1}\cdot\|g\|_{L_{\infty}}$. So, $\varphi(x)={}_{a}D_{x}^{\alpha}u(x)$ can be viewed as weak $\alpha$-th derivative of $u(x)$, and $I_{a+}^{\alpha}[L_1(\Omega)]$ is a Sobolev space.

Because $I_{a+}^{\alpha}[L_p(\Omega)]\subseteq I_{a+}^{\alpha}[L_1(\Omega)]$, $p\geq 1$, so, the discussions above $I_{a+}^{\alpha}[L_1(\Omega)]$ can be inherited by $I_{a+}^{\alpha}[L_p(\Omega)]$, $p\geq 1$. Therefore, we can say that $I_{a+}^{\alpha}[L_p(\Omega)]$ is a Sobolev space.

Since for $p\geq 1$,
$I_{a+}^{\alpha}[L_p(\Omega)]\hookrightarrow L_p(\Omega)$(Theorem 2.6 in\cite{Samko:93}), and (similar to Poincar\'{e} inequality\cite{Ervin:06})
\begin{equation*}
\|u(x)\|_{L_p(\Omega)}=\|{}_{a}I_{x}^{\alpha}\varphi(x)\|_{L_p(\Omega)}\le
\frac{(b-a)^\alpha}{\Gamma(\alpha+1)}
\|\varphi(x)\|_{L_p(\Omega)}
=\frac{(b-a)^\alpha}{\Gamma(\alpha+1)}
\|{}_{a}D_{x}^{\alpha}u(x)\|_{L_p(\Omega)},
\end{equation*}
the norm in (\ref{norm_redefined}) is indeed well-defined for $I_{a+}^{\alpha}[L_p(\Omega)]$ being as a Sobolev space.
\end{remark}

Based on Lemma \ref{lemma:2.2.2.1}, two corollaries can be obtained.
\begin{corollary}\label{corollary:2.2.2.1}
Let $n-1\leq \alpha<n$. If $u(x)\in I_{a+}^{\alpha}[L_p(\Omega)]$,
$v(x)\in I_{b-}^{\alpha}[L_q(\Omega)]$, with $\frac{1}{p}+\frac{1}{q}=1$, then
\begin{equation}\label{equation:2.2.8}
\langle {}_{a}D_{x}^{\alpha}u(x),v(x)\rangle=\langle u(x),{}_{x}D_{b}^{\alpha}v(x)\rangle,
\end{equation}
where $\langle \cdot,\cdot \rangle$ is the duality paring of $L_p(\Omega)$ and $L_q(\Omega)$.
\end{corollary}
\begin{proof}
If
$v(x)\in I_{b-}^{\alpha}[L_q(\Omega)]$, then there exists $\varphi_2(x)={}_{x}D_{b}^{\alpha}v(x)\in L_q(\Omega)$, such that $v(x)={}_{x}I_{b}^\alpha \varphi_2(x)$. So, by Lemma \ref{lemma:2.2.2.1},
$$\langle {}_{a}D_{x}^{\alpha}u(x),v(x)\rangle
=\langle {}_{a}D_{x}^{\alpha}u(x),{}_{x}I_{b}^\alpha \varphi_2(x)\rangle
=\langle {}_{a}I_{x}^\alpha{}_{a}D_{x}^{\alpha}u(x), \varphi_2(x)\rangle
=\langle u(x), {}_{x}D_{b}^{\alpha}v(x)\rangle.$$
\end{proof}

\begin{corollary}\label{corollary:2.2.1}
If $u(x)\in I_{a+}^\alpha[L_1(\Omega)]$, then
$${}_{a}D_{x}^{\alpha}u(x)=\textbf{D}^{\alpha}u(x)
:=D{}_{a}I_{x}^{n-\alpha}D^{n-1}u(x), \quad x\in \Omega.$$
In particular, we have
$${}_{a}D_{x}^{\alpha}u(x)=\textbf{D}^{\alpha}u(x)
:=D{}_{a}I_{x}^{2-\alpha}Du(x), \quad x\in \Omega,$$
for $1\leq\alpha<2$.
\end{corollary}
\begin{proof}
%
%

If $u(x)\in I_{a+}^\alpha[L_1(\Omega)]$, then there exists $\varphi(x)\in L_1(\Omega)$, such that
$u(x)={}_{a}I_{x}^{\alpha}\varphi(x)$, and ${}_{a}D_{x}^{\alpha}u(x)=\varphi(x)$. So,
$$
\begin{array}{lll}
D{}_{a}I_{x}^{n-\alpha}D^{n-1}u(x)
&=&D{}_{a}I_{x}^{n-\alpha}D^{n-1}{}_{a}I_{x}^{\alpha}\varphi(x)\\
&=&D{}_{a}I_{x}^{n-\alpha}D^{n-1}{}_{a}I_{x}^{n-1}I_{x}^{\alpha-(n-1)}\varphi(x)\\
&=&D{}_{a}I_{x}^{n-\alpha}I_{x}^{\alpha-(n-1)}\varphi(x)\\
&=&D{}_{a}I_{x}^{1}\varphi(x)\\
&=&\varphi(x).
\end{array}
$$
\end{proof}

\begin{remark}\label{remark:1}
The fractional operator $\textbf{D}^{\alpha}u(x):=D{}_{a}I_{x}^{n-\alpha}D^{n-1}u(x)$ is discussed in~\cite{Ervin:16}. The authors in~\cite{Ervin:16} believe that the abnormal diffusion problem with the operator defined in this way, and with normal Dirichlet boundary conditions, physically make sense.
\end{remark}

Besides the nonlocalness, the main point of Riemann-Liouville operator lies on the singularity at the boundary points. The authors in~\cite{Samko:93} make a contribution on recovering delicately the effect of Riemann-Liouville integral operator makes on the boundary point, which helps us to remove the ambiguous feeling about it. After relabeling and rearrangement, we list the embedding properties of fractional integral operator here:
\begin{lemma}[\rm Theorem 2.6, Theorem 3.5, Theorem 3.6 in~\cite{Samko:93}]\label{lemma:2.2.4}
\begin{enumerate}
\renewcommand{\labelenumi}{(\roman{enumi})}
\item
If $p\ge 1$, then
$$I_{a+}^\alpha[L_p(\Omega)]\hookrightarrow L_p(\Omega),$$
and
%
\begin{equation}\label{equation:2.2.9}
\|{}_{a}I_{x}^\alpha\varphi(x)\|_{p}\le \frac{(b-a)^\alpha}{\Gamma(\alpha+1)}\|\varphi(x)\|_{p}~~~\forall \varphi(x)\in L_p(\Omega).
\end{equation}

\item
If $0<\alpha<1$ and $1<p<1/\alpha$, then
$$I_{a+}^\alpha[L_p(\Omega)]\hookrightarrow L_q(\Omega), \textrm{~where~} q=\frac{p}{1-\alpha p}.$$
\item
If $p>1/\alpha$ and $\alpha-1/p \notin \mathbb{N}^{+}$, then $$I_{a+}^\alpha[L_p(\Omega)]\hookrightarrow C^{n_0,\alpha-1/p-n_0}(\Omega),$$
and
$${}_{a}I_{x}^\alpha\varphi(x)=o\left((x-a)^{\alpha-1/p}\right)~\textrm{as}~x\rightarrow a$$ for any $\varphi(x)\in L_p(\Omega)$,
where $n_0=\lfloor \alpha-1/p \rfloor$, $C^{m,\gamma}(\Omega)$ is a H\"{o}lder space defined by
\begin{equation}\label{equation:2.2.10}
\begin{array}{lll}
C^{m,\gamma}(\Omega)
&=&\left\{f:f(x)\in C^{m}(\Omega) \textrm{~and~}\exists~A>0 ~s.t. ~
\frac{|D^m f(x)-D^m f(y)|}{|x-y|^\gamma}\le A~ ~\forall x,y\in\Omega\right\}.
\end{array}
\end{equation}
\item
If $0<1/p<\alpha<1+1/p$, then
$$I_{a+}^\alpha[ L_p(\Omega)]\hookrightarrow c^{0,\alpha-1/p}(\Omega),$$
and
$${}_{a}I_{x}^\alpha\varphi(x)=o\left((x-a)^{\alpha-1/p}\right)~\textrm{as}~x\rightarrow a$$ for any $\varphi(x)\in L_p(\Omega)$,
where $c^{m,\gamma}(\Omega)$ 
is defined by
\begin{equation}\label{equation:2.2.11}
c^{m,\gamma}(\Omega)=\left\{f:f(x)\in C^{m}(\Omega) \textrm{~and~}
\frac{|D^m f(x)-D^m f(y)|}{|x-y|^\gamma}\rightarrow 0 ~\textrm{as~}x\rightarrow y\right\}.
\end{equation}
More specifically, given any $\epsilon>0$, for any $\varphi(x)\in L_p(\Omega)$, and any $x,\,y\in[a,b]$, there exist constants $c_1=c_1(\alpha)$, and $c_2=c_2(\alpha)$, such that
\begin{equation}\label{equation:2.2.11_1}
|{}_{a}I_{x}^{\alpha}\varphi(x)-{}_{a}I_{y}^{\alpha}\varphi(y)|
\leq c_1|x-y|^{\alpha}+c_2|x-y|^{\alpha-\frac{1}{p}}\cdot\epsilon.
\end{equation}
\end{enumerate}
\end{lemma}

By (\expandafter{\romannumeral1}) of Lemma \ref{lemma:2.2.4}, it is easy to see that the following results hold.
\begin{lemma}\label{lemma:2.2.5}
If $\alpha_1>\alpha_2\geq0$, then $I_{a+}^{\alpha_1}[L_p(\Omega)]\subseteq I_{a+}^{\alpha_2}[L_p(\Omega)]$.
\end{lemma}

\begin{lemma}\label{lemma:2.2.6}
Let $\alpha_1>0$, $\alpha_2>0$, $\alpha_1+\alpha_2=\alpha$. If
$u(x)\in I_{a+}^{\alpha}[L_p(\Omega)]$, then
\begin{equation}\label{equation:2.2.12}
{}_{a}D_{x}^{\alpha}u(x)={}_{a}D_{x}^{\alpha_1}{}_{a}D_{x}^{\alpha_2}u(x),
\end{equation}
and
\begin{equation}\label{equation:2.2.13}
{}_{a}D_{x}^{\alpha_2}u(x)\in I_{a+}^{\alpha_1}[L_p(\Omega)].
\end{equation}
\end{lemma}

\begin{remark}
A parallel result of Eq. (\ref{equation:2.2.12}) can also see in Theorem 2.13 of~\cite{Diethelm:10}.
\end{remark}

\begin{proof}
If $u(x)\in I_{a+}^{\alpha}[L_p(\Omega)]$, then there exists $\varphi(x)={}_{a}D_{x}^{\alpha}u(x)\in L_p(\Omega)$, such that
$$u(x)={}_{a}I_{x}^{\alpha}\varphi(x)={}_{a}I_{x}^{\alpha_2}[{}_{a}I_{x}^{\alpha_1}\varphi(x)].$$
By (\expandafter{\romannumeral1}) of Lemma \ref{lemma:2.2.4},
since ${}_{a}I_{x}^{\alpha_1}\varphi(x)\in L_p(\Omega)$, there is $u(x)\in I_{a+}^{\alpha_2}[L_p(\Omega)]$, and
$${}_{a}D_{x}^{\alpha_2}u(x)={}_{a}I_{x}^{\alpha_1}\varphi(x)\in I_{a+}^{\alpha_1}[L_p(\Omega)].$$
Also,
$${}_{a}D_{x}^{\alpha_1}{}_{a}D_{x}^{\alpha_2}u(x)
={}_{a}D_{x}^{\alpha_1}{}_{a}I_{x}^{\alpha_1}\varphi(x)
=\varphi(x)={}_{a}D_{x}^{\alpha}u(x).$$
\end{proof}

From definitions in (\ref{equation:2.2.10}) and (\ref{equation:2.2.11}), it is not difficult to see that $C^{m,\gamma}(\Omega)$ implies $c^{m,\gamma}(\Omega)$. What is more, the general Sobolev inequality in~\cite{Evans:10} shows that $W^{m+1,p}(\Omega)$, $m\geq 0$, $p>1$, is embedded into $C^{m,1-\frac{1}{p}}(\Omega)$, i.e.,
\begin{equation}\label{equation:2.2.7_1}
W^{m+1,p}(\Omega) \hookrightarrow C^{m,1-\frac{1}{p}}(\Omega),
\end{equation}
for $p>1$, $m\geq 0$, where $W^{m,p}(\Omega)$ is a classical Sobolev space, consists of all locally summable functions $\varphi: \Omega\longrightarrow \mathbb{R}$, such that for each $k$, $0\leq k\leq m$, $D^k\varphi(x)$ exists in the weak sense and belongs to $L_p(\Omega)$.

The following result shows indirectly a further relationship between the two H\"{o}lder spaces defined in (\ref{equation:2.2.10}) and (\ref{equation:2.2.11}).

\begin{lemma}\label{lemma:2.2.8}
If $0<1/p<\alpha<1+1/p$, then
\begin{equation}\label{equation:2.2.15}
I_{a+}^\alpha[ L_p(\Omega)]\subseteq W^{1,p'}(\Omega),
\end{equation}
where $p'=\frac{p}{1+p(1-\alpha)}$.
\end{lemma}

\begin{proof}
For any fixed $x\in[a,b]$, and any $0<\varepsilon_0<1/p$, let $\epsilon_0:=(x-a)^{\varepsilon_0}$. If $0<1/p<\alpha<1+1/p$, and $u\in I_{a+}^\alpha[ L_p(\Omega)]$, then by Eq. (\ref{equation:2.2.11_1}),
since $u(a)=0$, there exist $c_1:=c_1(\alpha)$ and $c_2:=c_2(\alpha)$, such that
$$|u(x)|\leq
c_1 (x-a)^{\alpha}+c_2(x-a)^{\alpha-\frac{1}{p}}\cdot\epsilon_0
=c_1 (x-a)^{\alpha}+c_2(x-a)^{\alpha-\frac{1}{p}+\varepsilon_0}.$$
Thereby, there exists a constant $C_0:=C_0(\alpha,a,b)>0$, such that
$$|u(x)|\leq C_0\cdot(x-a)^{\alpha-\frac{1}{p}+\varepsilon_0}.$$
Since $(x-a)^{\alpha-\frac{1}{p}+\varepsilon_0}\in W^{1,p'}(\Omega)$, we have $u(x)\in W^{1,p'}(\Omega).$
\end{proof}

\subsubsection{Intersection of Left-side and Right-side Spaces}\label{subsection:2.2.3}

The authors in~\cite{Samko:93} discussed the relationship between the $\alpha$-order left-side and right-side fractional integral spaces of $L_p(\Omega)$ functions. After flipping through pages and relabeling, we rearrange the results and list them as follows.

\begin{lemma}[\rm~\cite{Samko:93}, p.208 and p.337]\label{lemma:2.3.1}
When $0<\alpha<1/p$, and $1<p<\infty$, then
\begin{equation}\label{equation:2.3.1}
H^{\alpha,p}(\Omega)=I^{\alpha}[L_p(\Omega)]
:=I_{a+}^{\alpha}[L_p(\Omega)]=I_{b-}^{\alpha}[L_p(\Omega)],
\end{equation}
up to the equivalence of norms. When $1/p<\alpha<1/p+1$, then
\begin{equation}\label{equation:2.3.2}
H_0^{\alpha,p}(\Omega)
=I_{a+}^{\alpha}[L_p(\Omega)]\cap I_{b-}^{\alpha}[L_p(\Omega)],
\end{equation}
where
$$H_0^{\alpha,p}(\Omega)=\left\{f:f(x)\in H^{\alpha,p}(\Omega), \textrm{~and~} f(a)=f(b)=0\right\},$$
and
$$H^{\alpha,p}(\Omega)=\left\{f:\exists ~g(x)\in H^{\alpha,p}(\mathbb{R}),
~\textrm{s.t.}~ g(x)\big|_{\Omega}=f(x)\right\},$$
with the norm $\|f(x)\|_{H^{\alpha,p}(\Omega)}=\inf\|g(x)\|_{H^{\alpha,p}(\mathbb{R})}$.
\end{lemma}

Based on the above results, we can get two interesting results about the intersection of the image spaces of left-side and right-side Riemann-Liouville integrals of $L_p(\Omega)$.
\begin{theorem}\label{theorem:2.3.1}
Let $0\leq \alpha< \frac{1}{2}$. Then, there hold
\begin{equation}\label{equation:2.3.3}
I^{\alpha}[L_p(\Omega)]=L_2(\Omega),~p=\frac{2}{1+2\alpha},
\end{equation}
and
\begin{equation}\label{equation:2.3.4}
I^{\alpha}[L_2(\Omega)]=L_q(\Omega),~q=\frac{2}{1-2\alpha}.
\end{equation}
\end{theorem}

\begin{proof}
We only show the proof of (\ref{equation:2.3.3}). Equality (\ref{equation:2.3.4}) can be proved in a similar way.

On one hand, for any $\phi(x)\in L_p(\Omega)$, since $1<p=\frac{2}{1+2\alpha}<\frac{1}{\alpha}$, by (\expandafter{\romannumeral2}) of Lemma \ref{lemma:2.2.4}, we have ${}_{a}I_{x}^{\alpha}\phi(x)\in L_2(\Omega)$. Thus, by Lemma \ref{lemma:2.3.1}, $I^{\alpha}\phi(x)\subseteq L_2(\Omega)$.

One the other hand, for any $\psi(x)\in L_2(\Omega)$, we denote 
\begin{equation*}
\phi(x):={}_{a}D_{x}^{\alpha}\psi(x)=D{}_{a}I_{x}^{1-\alpha}\psi(x).
\end{equation*}
Since $\frac{1}{2}<1-\alpha<\frac{3}{2}$, according to (\expandafter{\romannumeral4}) of Lemma \ref{lemma:2.2.4} and Lemma \ref{lemma:2.2.8}, we have ${}_{a}I_{x}^{1-\alpha}\psi(x)\big|_{x=a}=0$, and ${}_{a}I_{x}^{1-\alpha}\psi(x)\in W^{1,p}(\Omega)$, $p=\frac{2}{1+2\alpha}$. Thus, $\phi(x)$ makes sense, and $\phi(x)={}_{a}D_{x}^{\alpha}\psi(x)\in L_p(\Omega)$, $p>1$.

Since ${}_{a}I_{x}^{1-\alpha}\psi(x)\in W^{1,p}(\Omega)\subseteq AC(\Omega)$, by Lemma \ref{lemma:2.2.1.1} and Lemma \ref{lemma:2.2.2.1}, there is $\psi(x)\in I^\alpha[L_1(\Omega)]$, and $\psi(x)={}_{a}I_{x}^{\alpha}{}_{a}D_{x}^{\alpha}\psi(x)$. Thus, $\psi(x)\in{}_{a}I_{x}^{\alpha}[L_p(\Omega)]$, because of ${}_{a}D_{x}^{\alpha}\psi(x)\in L_p(\Omega)$.

\end{proof}

\begin{theorem}\label{theorem:2.3.2}
Let $\frac{1}{2}<\alpha<\frac{3}{2}$. Then, there holds
\begin{equation}\label{equation:2.3.6}
\begin{array}{lll}
&&I_{a+}^{\alpha}[L_2(\Omega)]\cap I_{b-}^{\alpha}[L_2(\Omega)]\\
&=&\big\{
f:f(x)\in W^{1,q}(a,b),f(x)=o((x-a)^{\alpha-\frac{1}{2}})~\textrm{as}~x\rightarrow a,
f(x)=o((b-x)^{\alpha-\frac{1}{2}})~\textrm{as}~x\rightarrow b\big\},
\end{array}
\end{equation}
where $q=\frac{2}{3-2\alpha}$.
\end{theorem}

\begin{proof}
On one hand, if $v(x)\in I_{a+}^{\alpha}[L_2(\Omega)]\cap I_{b-}^{\alpha}[L_2(\Omega)]$, then
by (\expandafter{\romannumeral4}) of Lemma \ref{lemma:2.2.4}, we have $v(x)=o((x-a)^{\alpha-\frac{1}{2}})~\textrm{as}~x\rightarrow a$,
and $v(x)=o((b-x)^{\alpha-\frac{1}{2}})~\textrm{as}~x\rightarrow b$. By Lemma \ref{lemma:2.2.8}, we have
$v(x)\in W^{1,q}(\Omega)$, where $q=\frac{2}{3-2\alpha}$.

One the other hand, if $v(x)\in W^{1,q}(\Omega)$, $q=\frac{2}{3-2\alpha}$, then by Eq. (\ref{equation:2.2.7_1}), $v(x)$ is continuous, and $Dv(x)\in L_q(\Omega)$. What is more, if $v(x)=o((x-a)^{\alpha-\frac{1}{2}})~\textrm{as}~x\rightarrow a$, then by (\expandafter{\romannumeral2}) of Lemma \ref{lemma:2.2.4}, we can get $$\phi(x):=\,_{a}D_{x}^{\alpha}v(x)=D\,_{a}I_{x}^{1-\alpha}v(x)
=\,_{a}I_{x}^{1-\alpha}Dv(x)\in L_2(\Omega).$$
Thus, $\phi(x)$ makes sense. Also, since $v(x)$ is continuous, and $v(a)=0$,
$$\,_{a}I_{x}^{\alpha}\phi(x)
=\,_{a}I_{x}^{\alpha}\,_{a}I_{x}^{1-\alpha}Dv(x)=\,_{a}I_{x}^1\,Dv(x)
=v(x).$$
So, $v(x)\in I_{a+}^{\alpha}[L_2(\Omega)].$

Similarly, we can proof $v(x)\in I_{b-}^{\alpha}[L_2(\Omega)]$.
\end{proof}

\subsubsection{Tempered Fractional Operators}\label{subsection:2.2.4}

The tempered fractional calculus is a mathematical tool to describe the transition between normal and anomalous diffusions (or the anomalous motion in finite time or bounded physical space). Exponentially tempering the L\'{e}vy measure is a popular way
to make the moments of L\'{e}vy  distributions finite
in transport models. In spatially
tempered fractional Fokker-Planck equation~\cite{Cartea:07, Sabzikar:15}, the left and the right tempered
fractional derivative are defined, respectively, as
\begin{equation}\label{equation:2.2.4.1}
_{a}D_{x}^{\alpha,\lambda}u(x)
:=e^{-\lambda x}{{_{a}}D_{x}^{\alpha}}\big( e^{\lambda x}u(x)\big),
\end{equation}
and
\begin{equation}\label{equation:2.2.4.2}
_{x}D_b^{\alpha,\lambda}u(x)
:=e^{\lambda x}{{_{x}}D_b^{\alpha}}\big( e^{-\lambda x}u(x)\big),
\end{equation}
with $n-1\le \alpha<n,\,\lambda\ge 0$.

By above definitions, we can see that from the mathematical view, there is no essential difficulty when solving a tempered fractional problem by using finite difference methods, compared with solving the corresponding Riemann-Liouville one. However, as far as we know, there has no literature to show that tempered and Riemann-Liouville fractional problems are reciprocal in some sense when solving them by spectral methods or finite element methods.

In this subsection, we show that under the image space of fractional integrals of $L_p(\Omega)$ functions, tempered fractional operator is indeed equivalent to its corresponding Riemann-Liouville fractional operator, which is expected to make some efforts on theoretical support for relevant numerical methods.

We firstly borrow a result from~\cite{Samko:93}. After relabeling and rearrangement, we list it in the following way.

\begin{lemma}[\rm Lemma 10.3 in~\cite{Samko:93}]\label{lemma:2.2.3.1}
Let $\alpha>0$, $\sigma\in \mathbb{R}$, and $1\le p<\infty$. Then
\begin{equation}\label{equation:2.2.3.1}
e^{-\sigma x}{}_{a}I_{x}^{\alpha}e^{\sigma x}v(x)\in I_{a+}^{\alpha}[L_p(\Omega)]\Leftrightarrow
v(x)\in L_p(\Omega).
\end{equation}
\end{lemma}

\begin{corollary}\label{corollary:2.2.3.1}
Let $\alpha>0$, $\sigma\in \mathbb{R}$, and $1\le p<\infty$. Then
$$u(x)\in I_{a+}^{\alpha}[L_p(\Omega)]$$
if and only if
$$e^{\sigma x}u(x)\in I_{a+}^{\alpha}[L_p(\Omega)].$$
Also, there are positive constants $C_1$ and $C_2$, which depend on $\alpha$, $\sigma$, and
$a,~b$, such that
\begin{equation}\label{equation:2.2.3.3}
C_1\|u(x)\|_{I_{a+}^{\alpha}[L_p(\Omega)]}\leq
\|e^{\sigma x}u(x)\|_{I_{a+}^{\alpha}[L_p(\Omega)]}\leq
C_2 \|u(x)\|_{I_{a+}^{\alpha}[L_p(\Omega)]}.
\end{equation}
That is, there are positive constants $\tilde{C}_1$ and $\tilde{C}_2$, which depend on $\alpha$, $\sigma$, and
$a,~b$, such that
\begin{equation}\label{equation:2.2.3.3_1}
\tilde{C}_1\|{}_{a}D_{x}^{\alpha}u(x)\|_{L_p(\Omega)}\leq
\|e^{-\sigma x}{}_{a}D_{x}^{\alpha}\left(e^{\sigma x}u(x)\right)\|_{L_p(\Omega)}\leq
\tilde{C}_2 \|{}_{a}D_{x}^{\alpha}u(x)\|_{L_p(\Omega)}.
\end{equation}
\end{corollary}

\begin{proof}

If $u(x)\in I_{a+}^{\alpha}[L_p(\Omega)]$,
then there exists $\nu(x)\in L_p(\Omega)$, such that
\begin{equation*}
u(x)={}_{a}I_{x}^{\alpha}\nu(x)={}_{a}I_{x}^{\alpha}e^{-\sigma x} e^{\sigma x}\nu(x)
:={}_{a}I_{x}^{\alpha}e^{-\sigma x} \tilde{\nu}(x),
\end{equation*}
where $\tilde{\nu}(x)=e^{\sigma x}\nu(x)\in L_p(\Omega)$.
Therefore, $e^{\sigma x}u(x)=e^{\sigma x}{}_{a}I_{x}^{\alpha}e^{-\sigma x} \tilde{\nu}(x)$.
Then, by Lemma \ref{lemma:2.2.3.1}, we have $e^{\sigma x}u(x)\in I_{a+}^{\alpha}[L_p(\Omega)]$.
Similarly, we have $e^{-\sigma x}u(x)\in I_{a+}^{\alpha}[L_p(\Omega)]$.

If $e^{\sigma x}u(x)\in{}_{a}I_{x}^{\alpha}[L_p(\Omega)]$, then by the above analysis, we have
\begin{equation*}
e^{-\sigma x}\left(e^{\sigma x}u(x)\right)=u(x)\in I_{a+}^{\alpha}[L_p(\Omega)].
\end{equation*}

Now we prove there exists a constant $C$, which depends on $\alpha, \sigma, a, b$, such that
\begin{equation}\label{equation:2.2.3.4}
\|e^{\sigma x}{}_{a}I_{x}^{\alpha}e^{-\sigma x} \tilde{\nu}(x)\|_{I_{a+}^{\alpha}[L_p(\Omega)]}
\leq C\|\tilde{\nu}(x)\|_{L_p(\Omega)}.
\end{equation}
We only prove the case when $0<\alpha<1$.

Actually, according to the proof of Lemma 31.4 in~\cite{Samko:93}, we can get
$$e^{\sigma x}{}_{a}I_{x}^{\alpha}e^{-\sigma x} \tilde{\nu}(x)
={}_{a}I_{x}^{\alpha}\tilde{\nu}(x)+T_1\tilde{\nu}(x),$$
where
$$T_1\tilde{\nu}(x):=
\frac{1}{\Gamma(\alpha)}\int_{a}^x \frac{e^{\sigma(x-\xi)}-1}{(x-a)^{1-\alpha}}\cdot\tilde{\nu}(\xi)\,d\xi
=\frac{1}{\Gamma(\alpha)}\,{}_{a}I_{x}^{\alpha}V_1\tilde{\nu}(x),$$
and
$$V_1\tilde{\nu}(x):=\frac{\alpha}{\Gamma(1-\alpha)}
\int_{a}^x \tilde{\nu}(s)\,ds\int_{s}^x \frac{e^{\sigma(x-s)}-e^{\sigma(t-s)}}{(t-s)^{1-\alpha}(x-t)^{1+\alpha}}\,dt.$$
Since there exists a constant $c_1=c_1(\sigma,a,b)$, such that
$$|e^{\sigma(x-s)}-e^{\sigma(t-s)}|\leq c_1|x-t|,$$
then by the definition of Beta function
$B(z,w)=\int_{0}^1 \tau^{z-1}(1-\tau)^{w-1}\,d\tau~~(\Re(z)>0,\Re(w)>0)$, there exists another constant $c_2=c_2(\sigma,a,b)$, such that
$$\big|\int_{s}^x \frac{e^{\sigma(x-s)}-e^{\sigma(t-s)}}
{(t-s)^{1-\alpha}(x-t)^{1+\alpha}}\,dt\big|
\leq c_2\cdot\Gamma(\alpha)\Gamma(1-\alpha).$$
Thus,
$$\|V_1\tilde{\nu}\|_{L_p(\Omega)}\leq c_3(\sigma,a,b)\cdot\Gamma(\alpha+1)\|\tilde{\nu}\|_{L_p(\Omega)}.$$
Therefore, by Eqs. (\ref{norm_redefined}), (\ref{equation:2.2.9}), we have
$$\|e^{\sigma x}{}_{a}I_{x}^{\alpha}e^{-\sigma x} \tilde{\nu}(x)\|_{I_{a+}^{\alpha}[L_p(\Omega)]}
\leq \|{}_{a}I_{x}^{\alpha}\tilde{\nu}(x)\|_{I_{a+}^{\alpha}[L_p(\Omega)]}
+\|T_1\tilde{\nu}(x)\|_{I_{a+}^{\alpha}[L_p(\Omega)]}
\leq C\|\tilde{\nu}(x)\|_{L_p(\Omega)}.$$

Finally, by Eqs. (\ref{norm_redefined}), (\ref{equation:2.2.9}), and Lemma \ref{lemma:2.2.3.1},
we have
\begin{eqnarray*}
&&\|e^{\sigma x}u(x)\|_{I_{a+}^{\alpha}[L_p(\Omega)]}\\
&=&\|e^{\sigma x}{}_{a}I_{x}^{\alpha}e^{-\sigma x} \tilde{\nu}(x)\|_{I_{a+}^{\alpha}[L_p(\Omega)]}
\leq C\|\tilde{\nu}(x)\|_{L_p(\Omega)}
\leq C_2\|u(x)\|_{I_{a+}^{\alpha}[L_p(\Omega)]}.
\end{eqnarray*}
The lower bound
\begin{equation*}
C_1\|u(x)\|_{I_{a+}^{\alpha}[L_p(\Omega)]}
\leq \|e^{\sigma x}u(x)\|_{I_{a+}^{\alpha}[L_p(\Omega)]},
\end{equation*}
can be obtained in a similar way.
\end{proof}

For $\alpha>0$, and $\lambda>0$, denote
$$
{}_{a}I_{x}^{\alpha,\lambda}u(x)
:=e^{-\lambda x}{{_{a}}I_{x}^{\alpha}}\big( e^{\lambda x}u(x)\big),
$$
and
$$
{}_{x}I_b^{\alpha,\lambda}u(x)
:=e^{\lambda x}{{_{x}}I_b^{\alpha}}\big( e^{-\lambda x}u(x)\big).
$$
By Lemma \ref{lemma:2.2.3.1}, if $u(x)\in L_p(\Omega)$, $1\leq p<\infty$, then
${}_{a}I_{x}^{\alpha,\lambda}u(x)\in I_{a+}^{\alpha}[L_p(\Omega)]$, and ${}_{x}I_b^{\alpha,\lambda}u(x)\in I_{b-}^{\alpha}[L_p(\Omega)]$.

For any function $h(x)$, $x\in\Omega$, by means of
zero extension, it can also be viewed as a function defined on $\mathbb{R}$. So, if a function $g(x)\in L_p(a,b)$, then $\widehat{g}(\omega)\in
L_p(\mathbb{R})$,
where $\widehat{g}(\omega):=\mathcal{F}(g)(\omega)=\int_{\mathbb{R}}e^{i\omega x }g(x)\,dx$ is the Fourier transform of $g(x)$. We show the following result for tempered fractional integrals, which is similar to Lemma 2.4 in~\cite{Ervin:06}.

\begin{theorem}\label{theorem:2.2.3.1}

Let $0<\alpha<\frac{1}{2}$, and $\lambda\geq0$. If $u(x)\in L_p(\Omega)$, $p=\frac{2}{1+2\alpha}$,
then
\begin{equation}\label{equation:2.2.3.5}
\left(\,_{a}I_{x}^{\alpha,\lambda}u(x),\,_{x}I_{b}^{\alpha,\lambda}u(x)\right)
\geq \cos(\pi\alpha)\cdot\|\,_{a}I_{x}^{\alpha,\lambda}u(x)\|_2^2
=\cos(\pi\alpha)\cdot\|\,_{x}I_{b}^{\alpha,\lambda}u(x)\|_2^2.
\end{equation}
\end{theorem}

\begin{proof}

Firstly, by Theorem \ref{theorem:2.3.1} and Lemma \ref{lemma:2.2.3.1},
we can see that both $\,_{a}I_{x}^{\alpha,\lambda}u(x)$ and
$\,_{x}I_{b}^{\alpha,\lambda}u(x)$ belong to $L_2(\Omega)$.

By noticing that
$$
(\lambda+i\omega)^{-\alpha}=|\lambda^2+\omega^2|^{-\alpha/2}e^{-i\theta \alpha},
~~\textrm{and}~~
(\lambda-i\omega)^{-\alpha}=|\lambda^2+\omega^2|^{-\alpha/2}e^{i\theta \alpha},
$$
where $\tan\theta=\frac{\omega}{\lambda}$,
we can get
$$
\overline{(\lambda+i\omega)^{-\alpha}}=
\overline{(\lambda-i\omega)^{-\alpha}}\cdot e^{2i\theta \alpha},
~~\textrm{and~~} |\theta|\leq \frac{\pi}{2}.
$$
Therefore, by the equalities~\cite{Chen:15b}
$$
\mathcal{F}[\,_{a}D_{x}^{-\alpha,\lambda}g(x)](\omega)=
(\lambda-i\omega)^{-\alpha}\widehat{g}(\omega),
$$
and
$$
\mathcal{F}[\,_{x}D_{b}^{-\alpha,\lambda}g(x)](\omega)=
(\lambda+i\omega)^{-\alpha}\widehat{g}(\omega),
$$
we have
$$
\begin{array}{lll}
&&\left(\,_{a}I_{x}^{\alpha,\lambda}u(x),\,_{x}I_{b}^{\alpha,\lambda}u(x)\right)
\\
&=&\int_{-\infty}^{\infty}
(\lambda-i\omega)^{-\alpha}\widehat{f}(\omega)\cdot
\overline{(\lambda+i\omega)^{-\alpha}\widehat{f}(\omega)}\,d\omega
\\
&=&\int_{-\infty}^{\infty}
(\lambda-i\omega)^{-\alpha}\widehat{f}(\omega)\cdot e^{2i\theta\alpha}
\overline{(\lambda-i\omega)^{-\alpha}{\widehat{f}(\omega)}}\,d\omega
\\
&=&\int_{-\infty}^{\infty}\cos(2\theta\alpha)
\mathcal{F}\left({}_{a}I_{x}^{\alpha,\lambda}u(x)\right)(\omega)\,
\overline{\mathcal{F}\left({}_{a}I_{x}^{\alpha,\lambda}u(x)\right)(\omega)\,}d\omega
\\
&&+i\int_{0}^{\infty}\sin(2\theta\alpha)
\bigg[\mathcal{F}\left({}_{a}I_{x}^{\alpha,\lambda}u(x)\right)(\omega)\,
\overline{\mathcal{F}\left({}_{a}I_{x}^{\alpha,\lambda}u(x)\right)(\omega)\,}d\omega
\\
&&~~~~~~~~~~~~~~~~~~~~~~~-\mathcal{F}\left({}_{a}I_{x}^{\alpha,\lambda}u(x)\right)(-\omega)\,
\overline{\mathcal{F}\left({}_{a}I_{x}^{\alpha,\lambda}u(x)\right)(-\omega)\,}d\omega
\bigg]\,d\omega
\\
&=&\int_{-\infty}^{\infty}\cos(2\theta\alpha)
\mathcal{F}\left({}_{a}I_{x}^{\alpha,\lambda}u(x)\right)(\omega)\,
\overline{\mathcal{F}\left({}_{a}I_{x}^{\alpha,\lambda}u(x)\right)(\omega)\,}d\omega
\\
&&+i\int_{0}^{\infty}\sin(2\theta\alpha)
\bigg[\mathcal{F}\left({}_{a}I_{x}^{\alpha,\lambda}u(x)\right)(\omega)\,
\mathcal{F}\left({}_{a}I_{x}^{\alpha,\lambda}u(x)\right)(-\omega)\,d\omega
\\
&&~~~~~~~~~~~~~~~~~~~~~~~-\mathcal{F}\left({}_{a}I_{x}^{\alpha,\lambda}u(x)\right)(-\omega)\,
\mathcal{F}\left({}_{a}I_{x}^{\alpha,\lambda}u(x)\right)(\omega)\,d\omega
\bigg]\,d\omega
\\
&=&\int_{-\infty}^{\infty}\cos(2\theta\alpha)
\mathcal{F}\left({}_{a}I_{x}^{\alpha,\lambda}u(x)\right)(\omega)\,
\overline{\mathcal{F}\left({}_{a}I_{x}^{\alpha,\lambda}u(x)\right)(\omega)\,}d\omega.
\end{array}
$$
Since
$$
|2\theta\alpha|\leq\alpha\pi<\frac{\pi}{2},
~~~\textrm{and}~~\cos(2\theta\alpha)\geq\cos(\alpha\pi)>0,
$$
we have
$$
\left(\,_{a}I_{x}^{\alpha,\lambda}u(x),\,_{x}I_{b}^{\alpha,\lambda}u(x)\right)
\geq
\cos(\alpha\pi)\|{}_{a}I_{x}^{\alpha,\lambda}u(x)\|_2^2.
$$
\end{proof}

\section{Image Space of Riemann-Liouville Integrals on $W^{m,p}(\Omega)$ Space}\label{section:3}

In this section, we generalize the image spaces of Riemann-Liouville integrals of $L_p(\Omega)$ functions to the spaces that possess high regularity.

\begin{definition}\label{definition:3.0.1}
Let $\alpha>0$ and $1\le p<\infty$. Define
\begin{equation}\label{equation:3.0.1}
I_{a+}^\alpha[W^{m,p}(\Omega)]:=\left\{f:f(x)={}_{a}I_{x}^\alpha \varphi(x), \varphi(x)\in W^{m,p}(\Omega), x\in \Omega\right\},
\end{equation}
and
\begin{equation}\label{equation:3.0.2}
I_{b-}^{\alpha}[W^{m,p}(\Omega)]:=\left\{f:f(x)={}_{x}I_{b}^\alpha \varphi(x), \varphi(x)\in W^{m,p}(\Omega), x\in \Omega\right\}.
\end{equation}
\end{definition}
Similar to the discussions in the above section, $I_{a+}^\alpha[W^{m,p}(\Omega)]$ and $I_{b-}^\alpha[W^{m,p}(\Omega)]$ are also Sobolev spaces, equipped with the norms
\begin{equation}\label{equation:3.0.3}
\|u(x)\|_{I_{a+}^{\alpha}[W^{m,p}(\Omega)]}:=
\|{}_{a}D_{x}^{\alpha}u(x)\|_{W^{m,p}(\Omega)},
\end{equation}
and
\begin{equation}\label{equation:3.0.4}
\|v(x)\|_{I_{b-}^{\alpha}[W^{m,p}(\Omega)]}:=
\|{}_{x}D_{b}^{\alpha}v(x)\|_{W^{m,p}(\Omega)},
\end{equation}
respectively.

It is easy to see that, for $0<\alpha<1$, and any positive integer $m$, there is
$$I_{a+}^{\alpha+m}[L_p(\Omega)]\subseteq I_{a+}^\alpha[W^{m,p}(\Omega)],$$
which means, the space $I_{a+}^\alpha[W^{m,p}(\Omega)]$ we defined here is not a trivial generalization.

The characteristic of the space $I_{a+}^\alpha[W^{m,p}(\Omega)]$ is listed as below.
\begin{theorem}\label{theorem:3.0.1}
In order that $u(x)\in I_{a+}^\alpha[W^{m,p}(\Omega)]$, $n-1\le \alpha <n$, $1< p<\infty$, it is necessary and
sufficient that ${}_{a}I_{x}^{n-\alpha}u(x)\in W^{m+n,p}(\Omega)$,
and that
\begin{equation*}
{}_{a}D_{x}^{\alpha-k}u(x)\big|_{x=a}=0,~~k=1,2,\cdots,n.
\end{equation*}
\end{theorem}

\begin{proof}
\emph{Necessity.} Let $u(x)={}_{a}I_{x}^{\alpha}\varphi(x)$,
$\varphi\in W^{m,p}(\Omega)$. In view of the semi-group property, we have $${}_{a}I_{x}^{n-\alpha}u(x)={}_{a}I_{x}^{n}\varphi(x)\in W^{m+n,p}(\Omega).$$
They by the definitions of $AC(\Omega)$ as well as $AC^n(\Omega)$ in (\ref{equation:2.1.1})-(\ref{equation:2.1.2}), we have
$$D^l{}_{a}I_{x}^{n}\varphi(x)\big|_{x=a}
={}_{a}D_{x}^{\alpha-(n-l)}u(x)\big|_{x=a}=0,\quad l=0,1,\cdots,n-1.$$

\emph{Sufficiency.} If
$${}_{a}I_{x}^{n-\alpha}u(x)
\in W^{m+n,p}(\Omega),$$
then by Eq. (\ref{equation:2.2.7_1}),
$${}_{a}I_{x}^{n-\alpha}u(x)
\in C^{m+n-1,1-\frac{1}{p}}(\Omega)\subseteq AC^{n}(\Omega).$$
If also
$$D^l {}_{a}I_{x}^{n-\alpha}u(x)\big|_{x=a}=0,\quad l=0,1,\cdots,n-1,$$
then by Lemma \ref{lemma:2.2.1.1}, $u(x)\in I_{a+}^\alpha[L_1(\Omega)]$, i.e.,
there exists $\varphi(x)\in L_1(\Omega)$, such that
$$u(x)={}_{a}I_{x}^{\alpha}\varphi(x).$$
Actually, since
${}_{a}I_{x}^{n}\varphi(x)={}_{a}I_{x}^{n-\alpha}u(x)\in W^{m+n,p}(\Omega)$,
$\varphi(x)\in W^{m,p}(\Omega)$. Thus, $u(x)\in I_{a+}^\alpha[W^{m,p}(\Omega)]$.

\end{proof}

\subsection{Properties}\label{subsection:3.1}

Since $I_{a+}^\alpha[W^{m,p}(\Omega)]\subseteq I_{a+}^{\alpha}[L_p(\Omega)]$, all properties of $I_{a+}^{\alpha}[L_p(\Omega)]$ we discussed in above section can be inherited by the space $I_{a+}^\alpha[W^{m,p}(\Omega)]$. Besides these, the space $I_{a+}^\alpha[W^{m,p}(\Omega)]$ also owns its special properties.

Let $u(x)={}_{a}I_{x}^{\alpha}\varphi(x)$, $\varphi(x)\in W^{m,p}(\Omega)$. In fact, if $m\geq 1$, and $1<p<\infty$, then by Eq. (\ref{equation:2.2.7_1}), there is $W^{m,p}(\Omega)\in C^{m-1,1-\frac{1}{p}}(\Omega)$, which means $\varphi(x)$ is continuous up to its $(m-1)$-th derivative. So, for $k=1,2,\cdots,m$, there is~\cite{Podlubny:99}
\begin{equation}\label{equation:3.1.1}
D^k u(x)= D^k{}_{a}I_{x}^{\alpha}\varphi(x)
={}_{a}I_{x}^{\alpha}[D^k\,\varphi(x)]
+\sum_{j=1}^{k}\frac{(x-a)^{\alpha-j}}{\Gamma(1+\alpha-j)}
\cdot D^{k-j}\varphi(a).
\end{equation}
In short, if $u(x)=I_{a+}^\alpha[W^{m,p}(\Omega)]$, $m\geq 1$, and $1<p<\infty$, then for $k=1,2,\cdots,m$, there is
$$D^k u(x)=v_1(x)+v_2(x),$$
where
$$v_1(x)=\sum_{j=1}^{k}\frac{(x-a)^{\alpha-j}}{\Gamma(1+\alpha-j)}
\cdot [D^{k-j}{}_{a}D_{x}^{\alpha}u(x)]\big|_{x=a},$$
and
$$v_2(x)\in I_{a+}^\alpha[W^{m-k,p}(\Omega)].$$

\begin{remark}\label{remark:3.0.1}
It should be noted that, different from the embedding properties of $I_{a+}^{\alpha}[L_p(\Omega)]$ in Lemma \ref{lemma:2.2.4}, one can observe from Eq. (\ref{equation:3.1.1}) that,
in general,
$$I_{a+}^\alpha[W^{m,p}(\Omega)]\nsubseteq W^{m,p}(\Omega),$$
because of the summation term in Eq. (\ref{equation:3.1.1}). However, we shall see later that, for $m\geq 0$, and $1<p<\infty$,
$$I_{a+}^\alpha[W^{m,p}(\Omega)]\cap I_{b-}^\alpha[W^{m,p}(\Omega)]
\subseteq W^{m,p}(\Omega)$$
indeed holds.
\end{remark}

\subsubsection{Properties}\label{section:3.1.1}

Define $P_{N}(\Omega)$ as the polynomials spaces of degree less than or equal to $N$ (if $N=-1$, we
define $P_{-1}[-1,1]:=\{0\})$.

We begin discussing the properties of $I_{a+}^{\alpha}[W^{m,p}(\Omega)]$ from the special case when $p=2$ and $W^{m,p}(\Omega)=H^{m}(\Omega)$.

\begin{theorem}\label{theorem:3.1.1}
Let $0<\alpha<1$. Then
$u(x)\in I_{a+}^{\alpha}[H^m(\Omega)],~~m\geq 0,$
is equivalent to
\begin{equation*}
u(x)=u_1(x)+u_2(x),
\end{equation*}
where
\begin{subequations}
\begin{equation*}
u_1(x)\in (x-a)^{\alpha}\cdot P_{m-1}(\Omega),
\end{equation*}
and
$$u_2(x)\in
\left\{
\begin{array}{lll}
W^{m,q_1}(\Omega),& q_1=\frac{2}{1-2\alpha},&
\alpha\in(0,\frac{1}{2}),\vspace{5pt}\\
W^{m+1,q_1'}(\Omega),& q_1'=\frac{2}{3-2\alpha},&
\alpha\in(\frac{1}{2},1),
\end{array}
\right.
$$
with
\begin{equation*}
u_2(x)=o\left((x-a)^{m+\alpha-\frac{1}{2}}\right)
~\textrm{as~} x\rightarrow a.
\end{equation*}
\end{subequations}
\end{theorem}

\begin{proof}
We only show the proof when $\alpha\in(0,\frac{1}{2})$. The case of $\alpha\in(\frac{1}{2},1)$ can be obtained in a similar way.

We prove the result for $m\geq 1$, since when $m=0$, $P_{-1}(\Omega)=\{0\}$, the result is immediate by Theorem \ref{theorem:2.3.1}.

\emph{Necessity:}

If $u(x)\in I_{a+}^{\alpha}[H^m(\Omega)]$, then there exists
$v(x)= {}_{a}D_{x}^{\alpha} u(x)\in H^m(\Omega)$.
Also, similar to Eq. (\ref{equation:3.1.1}), we can get
$$
v(x)
={}_{a}I_{x}^{m} D^mv(x)+\sum_{j=0}^{m-1} \frac{(x-a)^j}{\Gamma(1+j)}\cdot D^{j}v(a)
:=v_2(x)+v_1(x),
$$
and
\begin{equation*}
u(x)
={}_{a}I_{x}^{m+\alpha} D^mv_2(x)
+(1+x)^{\alpha}p_{m-1}(x)
:=u_2(x)+u_1(x),
\end{equation*}
where $p_{m-1}(x)\in P_{m-1}(\Omega)$.

Since $D^mv(x)\in L_2(\Omega)$, by (\expandafter{\romannumeral3}) of Lemma \ref{lemma:2.2.4}, we have
\begin{equation*}
v_2(x)={}_{a}I_{x}^{m}D^mv(x) \in C^{m-1,\frac{1}{2}}(\Omega),
~~~\textrm{and}~~ v_2(x)=o\left((x-a)^{m-\frac{1}{2}}\right)~~\textrm{as~~} x\rightarrow a;
\end{equation*}
also $u_2(x)={}_{a}I_{x}^{m+\alpha}D^mv(x)$ satisfies
$$u_2(x)=o\left((x-a)^{m+\alpha-\frac{1}{2}}\right)~~\textrm{as~~} x\rightarrow a.$$
Thus,
\begin{equation}\label{equation3.1.4.18}
D^kv_2(x)\big|_{x=a}=0,~~k=0,1,\cdots,m-1.
\end{equation}
Then we have
\begin{equation}\label{equation3.1.4.19}
D^j u_2(x)=D^j {}_{a}I_{x}^{\alpha}v_2(x)
={}_{a}I_{x}^{\alpha}\left(D^jv_2(x)\right),~~j=0,1,\cdots,m.
\end{equation}

In addition, since $D^j v_2(x)\in L_2(\Omega)$, it can be yielded from (\expandafter{\romannumeral2}) of Lemma \ref{lemma:2.2.4} that
\begin{equation*}
D^j u_2(x)\in L_{q_1}(\Omega),~~j=0,1,\cdots,m,
\end{equation*}
where $q_1=\frac{2}{1-2\alpha}$. That is, $u_2(x)\in W^{m,q_1}(\Omega)$.

\emph{Sufficiency:}

If $u_2(x)\in W^{m,q_1}(\Omega)$, $q_1=\frac{2}{1-2\alpha}$, then by Eq. (\ref{equation:2.2.7_1}),
$u_2(x)\in C^{m-1,\frac{1}{2}+\alpha}(\Omega)$. Combined with
\begin{equation*}
u_2(x)=o\left((x-a)^{m+\alpha-\frac{1}{2}}\right)
~\textrm{as~} x\rightarrow a,
\end{equation*}
we obtain
\begin{equation*}
D^ju_2(x)\big|_{x=a}=0,~~j=0,1,\cdots,m-1,
\end{equation*}
and
\begin{equation*}
D^k v_2(x):=D^k {}_{a}D_{x}^{\alpha}u_2(x)
={}_{a}I_{x}^{1-\alpha}\left(D^{k+1}u_2(x)\right),\quad k=0,1,\cdots,m-1.
\end{equation*}
Then by Lemma \ref{lemma:2.2.8}, we get
$$D^k v_2(x)\in H^1(\Omega),\quad k=0,1,\cdots,m-1.$$
Thus,
$v_2(x)={}_{a}D_{x}^{\alpha}u_2(x)=D{}_{a}I_{x}^{1-\alpha}u_2(x)\in H^m(\Omega)$.
Thereby,
$${}_{a}I_{x}^{1-\alpha}u_2(x)\in H^{m+1}(\Omega).$$
Since $u_2(x)$ is a continuous function,
${}_{a}I_{x}^{1-\alpha}u_2(x)|_{x=a}=0.$
Finally, by Theorem \ref{theorem:3.0.1}, we obtain
$$u(x)=u_1(x)+u_2(x)\in I_{a+}^{\alpha}[H^{m}(\Omega)].$$
\end{proof}

Another interesting result is listed in the following theorem.
\begin{theorem}\label{theorem:3.1.2}
\begin{enumerate}
\renewcommand{\labelenumi}{(\roman{enumi})}
\item
Let $0<\alpha<\frac{1}{2}$, and $p_2=\frac{2}{1+2\alpha}$. Then
$u(x)\in I_{a+}^{\alpha}[W^{m,p_2}(\Omega)],~m\geq 0,$
is equivalent to
\begin{equation*}
u(x)=u_1(x)+u_2(x),
\end{equation*}
where
\begin{subequations}
\begin{equation*}
u_1(x)\in (x-a)^{\alpha}\cdot P_{m-1}(\Omega),
\end{equation*}
and
$$u_2(x)\in H^m(\Omega),$$
with
\begin{equation*}
u_2(x)=o\left((x-a)^{m-\frac{1}{2}}\right)
~\textrm{as~} x\rightarrow a.
\end{equation*}

\end{subequations}
\item
Let $\frac{1}{2}<\alpha<1$, and $p_2'=\frac{2}{2\alpha-1}$. Then
$u(x)\in I_{a+}^{\alpha}[W^{m,p_2'}(\Omega)],~m\geq 0,$
is equivalent to
\begin{equation*}
u(x)=u_1(x)+u_2(x),
\end{equation*}
where
\begin{subequations}
\begin{equation*}
u_1(x)\in (x-a)^{\alpha}\cdot P_{m-1}(\Omega),
\end{equation*}
and
$$u_2(x)\in H^{m+1}(\Omega),$$
with
\begin{equation*}
u_2(x)=o\left((x-a)^{m+\frac{1}{2}}\right)
~\textrm{as~} x\rightarrow a.
\end{equation*}
\end{subequations}
\end{enumerate}
\end{theorem}

\begin{proof}
We only show the proof when $\alpha\in(0,\frac{1}{2})$. The case of $\alpha\in(\frac{1}{2},1)$ can be obtained in a similar way.

We prove the result for $m\geq 1$, since when $m=0$, $P_{-1}(\Omega)=\{0\}$, the result is immediate by Theorem \ref{theorem:2.3.1}.

\emph{Necessity:}

If $u(x)\in I_{a+}^{\alpha}[W^{m,p_2}(\Omega)]$, then there exists
$v(x)= {}_{a}D_{x}^{\alpha} u(x)\in W^{m,p_2}(\Omega)$.
Also,
$$
v(x)
={}_{a}I_{x}^{m} D^mv(x)+\sum_{j=1}^{m}\frac{(x-a)^{m-j}}{\Gamma(1+m-j)}
\cdot D^{m-j}v(a)
:=v_2(x)+v_1(x),
$$
and
\begin{equation*}
u(x)
={}_{a}I_{x}^{m+\alpha} D^mv_2(x)
+(1+x)^{\alpha}p_{m-1}(x)
:=u_2(x)+u_1(x),
\end{equation*}
where $p_{m-1}(x)\in P_{m-1}(\Omega)$.

Since $D^mv(x)\in L_{p_2}(\Omega)$, $v_2(x)={}_{a}I_{x}^{m}D^mv(x) \in W^{m,p_2}(\Omega)$. By (\expandafter{\romannumeral4}) of Lemma \ref{lemma:2.2.4}, we have
\begin{equation*} v_2(x)=o\left((x-a)^{m-\alpha-\frac{1}{2}}\right)~~\textrm{as~~} x\rightarrow a;
\end{equation*}
also $u_2(x)={}_{a}I_{x}^{m+\alpha}D^mv(x)$ satisfies
$$u_2(x)=o\left((x-a)^{m-\frac{1}{2}}\right)~~\textrm{as~~} x\rightarrow a.$$
Thus,
\begin{equation}\label{equation3.1.4.18}
D^kv_2(x)\big|_{x=a}=0,~~k=0,1,\cdots,m-1.
\end{equation}
Then we have
\begin{equation}\label{equation3.1.4.19}
D^j u_2(x)=D^j {}_{a}I_{x}^{\alpha}v_2(x)
={}_{a}I_{x}^{\alpha}\left(D^jv_2(x)\right),~~j=0,1,\cdots,m.
\end{equation}

In addition, since $D^j v_2(x)\in L_{p_2}(\Omega)$, it can be yielded from Eq. (\ref{equation:2.3.3}) that
\begin{equation*}
D^j u_2(x)\in L_2(\Omega),~~j=0,1,\cdots,m.
\end{equation*}
That is, $u_2(x)\in H^{m}(\Omega)$.

\emph{Sufficiency:}

If $u_2(x)\in H^{m}(\Omega)$, then by Eq. (\ref{equation:2.2.7_1}),
$u_2(x)\in C^{m-1,\frac{1}{2}}(\Omega)$. Combined with
\begin{equation*}
u_2(x)=o\left((x-a)^{m-\frac{1}{2}}\right)
~\textrm{as~} x\rightarrow a,
\end{equation*}
we obtain
\begin{equation*}
D^ju_2(x)\big|_{x=a}=0,~~j=0,1,\cdots,m-1,
\end{equation*}
and
\begin{equation*}
D^k v_2(x):=D^k {}_{a}D_{x}^{\alpha}u_2(x)
={}_{a}I_{x}^{1-\alpha}\left(D^{k+1}u_2(x)\right),\quad k=0,1,\cdots,m-1.
\end{equation*}
Then by Lemma \ref{lemma:2.2.8}, we get
$$D^k v_2(x)\in W^{1,p_2}(\Omega),\quad k=0,1,\cdots,m-1.$$
Thus,
$v_2(x)={}_{a}D_{x}^{\alpha}u_2(x)=D{}_{a}I_{x}^{1-\alpha}u_2(x)\in W^{m,p_2}(\Omega)$.
Thereby,
$${}_{a}I_{x}^{1-\alpha}u_2(x)\in W^{m+1,p_2}(\Omega).$$
Since $u_2(x)$ is a continuous function,
$${}_{a}I_{x}^{1-\alpha}u_2(x)|_{x=a}=0.$$
Finally, by Theorem \ref{theorem:3.0.1}, we obtain
$$u(x)=u_1(x)+u_2(x)\in I_{a+}^{\alpha}[W^{m,p_2}(\Omega)].$$
\end{proof}

More generally, the following theorem holds. Here we omit the proof, since it is similar to the above analysis.

\begin{theorem}\label{theorem:3.1.3}
Let $0<\alpha<1$ and $1<p<\infty$. Then
$u(x)\in I_{a+}^{\alpha}[W^{m,p}(\Omega)],~~m\geq 0,$
is equivalent to
\begin{equation*}
u(x)=u_1(x)+u_2(x),
\end{equation*}
where
\begin{subequations}
\begin{equation*}
u_1(x)\in (x-a)^{\alpha}\cdot P_{m-1}(\Omega),
\end{equation*}
and
$$u_2(x)\in
\left\{
\begin{array}{lll}
W^{m,q}(\Omega),& \frac{1}{q}=\frac{1}{p}-\alpha,&
\alpha p<1,\vspace{5pt}\\
W^{m+1,q'}(\Omega),& \frac{1}{q'}=\frac{1}{p}+1-\alpha,&
\alpha p>1,
\end{array}
\right.
$$
with
\begin{equation*}
u_2(x)=o\left((x-a)^{m+\alpha-\frac{1}{p}}\right)
~\textrm{as~} x\rightarrow a.
\end{equation*}
\end{subequations}
\end{theorem}

\begin{remark}
Theorem \ref{theorem:3.1.3} tells the main difference between the space of fractional integrals of $W^{m,q}(\Omega)$ and the fractional derivative space which is defined as the closures of $C_0^{\infty}(\Omega)$ functions under the norm of fractional Sobolev spaces on $\mathbb{R}$: the behavior of the functions in $I_{a+}^{\alpha}[W^{m,p}(\Omega)]$ and $I_{b-}^{\alpha}[W^{m,p}(\Omega)]$ at both the nearby of the boundary as well as the points in the domain are clear.
\end{remark}

\subsubsection{Intersection of Left-side and Right-side Spaces}\label{section:3.1.2}
Let $0<\alpha<1$, $1<p<\infty$, and $m\geq 0$. It is not difficulty to see that, for any $p_{m-1}(x)\in P_{m-1}(\Omega)$, if
$$(b-x)^{\alpha}p_{m-1}(x)=o\left((x-a)^{m+\alpha-\frac{1}{p}}\right)
~\textrm{as~} x\rightarrow a ~\textrm{or~} x\rightarrow b,$$
only if $p_{m-1}(x)\equiv 0$.
So, with the help of Theorems \ref{theorem:3.1.1}-\ref{theorem:3.1.3}, we directly list the following result about the intersection of left-side and right-side spaces.
\begin{theorem}\label{theorem:3.1.2.1}
Let $0<\alpha<1$ and $1<p<\infty$. Then
$u(x)\in I_{a+}^{\alpha}[W^{m,p}(\Omega)]\cap I_{b-}^{\alpha}[W^{m,p}(\Omega)],~~m\geq 0,$
is equivalent to
$$u(x)\in
\left\{
\begin{array}{lll}
W^{m,q}(\Omega),& \frac{1}{q}=\frac{1}{p}-\alpha,&
\alpha p<1,\vspace{5pt}\\
W^{m+1,q'}(\Omega),& \frac{1}{q'}=\frac{1}{p}+1-\alpha,&
\alpha p>1.
\end{array}
\right.
$$
with
\begin{equation*}
u(x)=o\left((x-a)^{m+\alpha-\frac{1}{p}}\right)
~\textrm{as~} x\rightarrow a,\textrm{~and~}
u(x)=o\left((b-x)^{m+\alpha-\frac{1}{p}}\right)
~\textrm{as~} x\rightarrow b.
\end{equation*}
Especially:
\begin{enumerate}
\renewcommand{\labelenumi}{(\roman{enumi})}
\item
if $0\leq\alpha<\frac{1}{2}$ and $q_1=\frac{2}{1-2\alpha}$, then
\begin{eqnarray*}
&&I_{a+}^{\alpha}[H^{m}(\Omega)]\cap I_{b-}^{\alpha}[H^{m}(\Omega)]
\nonumber\\
&=&\bigg\{
f:f(x)\in W^{m,q_1}(\Omega),
f(x)=o((x-a)^{m+\alpha-\frac{1}{2}})~\textrm{as}~x\rightarrow a,
f(x)=o((b-x)^{m+\alpha-\frac{1}{2}})~\textrm{as}~x\rightarrow b\bigg\}.
\end{eqnarray*}
\item
if $\frac{1}{2}<\alpha<1$ and $q_1'=\frac{2}{3-2\alpha}$, then
\begin{eqnarray*}
&&I_{a+}^{\alpha}[H^{m}(\Omega)]\cap I_{b-}^{\alpha}[H^{m}(\Omega)]
\nonumber\\
&=&\bigg\{
f:f(x)\in W^{m+1,q_1'}(\Omega),
f(x)=o((x-a)^{m+\alpha-\frac{1}{2}})~\textrm{as}~x\rightarrow a,
f(x)=o((b-x)^{m+\alpha-\frac{1}{2}})~\textrm{as}~x\rightarrow b\bigg\}.
\end{eqnarray*}
\item
if $0<\alpha<\frac{1}{2}$, and $p_2=\frac{2}{1+2\alpha}$, then
\begin{eqnarray*}
&&I_{a+}^{\alpha}[W^{m,p_2}(\Omega)]\cap I_{b-}^{\alpha}[W^{m,p_2}(\Omega)]
\nonumber\\
&=&\bigg\{
f:f(x)\in H^{m}(\Omega),
f(x)=o((x-a)^{m-\frac{1}{2}})~\textrm{as}~x\rightarrow a,
f(x)=o((b-x)^{m-\frac{1}{2}})~\textrm{as}~x\rightarrow b\bigg\}.
\end{eqnarray*}
\item
if $\frac{1}{2}<\alpha<1$, and $p_2'=\frac{2}{2\alpha-1}$, then
\begin{eqnarray*}
&&I_{a+}^{\alpha}[W^{m,p_2'}(\Omega)]\cap I_{b-}^{\alpha}[W^{m,p_2'}(\Omega)]
\nonumber\\
&=&\bigg\{
f:f(x)\in H^{m+1}(\Omega),
f(x)=o((x-a)^{m+\frac{1}{2}})~\textrm{as}~x\rightarrow a,
f(x)=o((b-x)^{m+\frac{1}{2}})~\textrm{as}~x\rightarrow b\bigg\}.
\end{eqnarray*}
\end{enumerate}
\end{theorem}

\begin{remark}\label{remark:3.1.2.1}
As we mentioned in Remark \ref{remark:3.0.1}, one can easily check by Theorem \ref{theorem:3.1.2.1} that, for $m\geq 0$, and $1<p<\infty$,
$$I_{a+}^\alpha[W^{m,p}(\Omega)]\cap I_{b-}^\alpha[W^{m,p}(\Omega)]
\subseteq W^{m,p}(\Omega).$$
What is more, by Eq. (\ref{equation:3.1.1}), if $u(x)\in I_{a+}^\alpha[W^{m,p}(\Omega)]\cap I_{b-}^\alpha[W^{m,p}(\Omega)]$, then
for $k=0,1,\cdots,m$,
\begin{equation}\label{equation:3.1.2.1}
\|D^k u\|_{L_p(\Omega)}
\leq C\|D^k u\|_{I_{a+}^\alpha[L_p(\Omega)]}
\leq C|{}_{a}D_{x}^\alpha u|_{W^{k,p}(\Omega)},
\end{equation}
\begin{equation}\label{equation:3.1.2.2}
\|D^k u\|_{L_p(\Omega)}
\leq C\|D^k u\|_{I_{b-}^\alpha[L_p(\Omega)]}
\leq C|{}_{x}D_{b}^\alpha u|_{W^{k,p}(\Omega)},
\end{equation}
because of $[D^j {}_{a}D_{x}^\alpha u](x)|_{x=a}
=[D^j {}_{x}D_{b}^\alpha u](x)|_{x=b}=0$, for $j=0,1,\cdots,m-1.$ Therefore, we have
$$\|u\|_{W^{m,p}(\Omega)}\leq C\|{}_{a}D_{x}^\alpha u\|_{W^{m,p}(\Omega)},$$
and
$$\|u\|_{W^{m,p}(\Omega)}\leq C\|{}_{x}D_{b}^\alpha u\|_{W^{m,p}(\Omega)}.$$
\end{remark}

\subsubsection{Approximation Property}\label{section:3.1.3}

The idea we want to apply the image spaces defined above in numerical methods to fractional differential equations is that, instead of approximating the solution (which belongs to some image space such as $I_{a+}^{\beta}[W^{m,p}(\Omega)]$) in a classical finite-dimensional space, we approximate it in $I_{a+}^{\beta}[P_N(\Omega)]$ for some given value $\beta$. That is, we use a polynomial to approximate a function in classical Sobolev space $W^{m,p}(\Omega)$ just exactly as the classical way, and then take it as the integrand of a Riemann-Liouville fractional operator. As a matter of fact, when applied in numerical methods to fractional differential equations, the whole process can be guaranteed by the two formulae in Remark \ref{remark:4}. In this subsection, we show some basic corresponding approximation results.

Denote $$I_{a+}^{\alpha}\left[P_{N}(\Omega)\right]:=\left\{f:f(x)={}_{a}I_{x}^\alpha \varphi(x), \varphi(x)\in P_{N}(\Omega), x\in\Omega\right\}.$$
It is obvious that
$$I_{a+}^{\alpha}\left[P_{N}(\Omega)\right] \subseteq I_{a+}^{\alpha}[L_2(\Omega)].$$
Denote $\Pi_N$ as the orthogonal projection operator from  $L_2(\Omega)$ onto $P_{N}(\Omega)$. Denote $I_N$ as the interpolation operator from from  $L_2(\Omega)$ onto $P_{N}(\Omega)$. We recall that for any function $v(x)\in H^m(\Omega)$, there exit constants $C_1=C_1(\Omega,m)$ and $C_2=C_2(\Omega,m)$, such that~\cite{Hesthaven:07}
\begin{equation}\label{equation:3.1.3.1}
\|v(x)-\Pi_N v(x)\|_{L_2(\Omega)}\leq C_1N^{-m}|v|_{H^m(\Omega)},
\end{equation}
and
\begin{equation}\label{equation:3.1.3.2}
\|v(x)-I_N v(x)\|_{L_2(\Omega)}\leq C_2 N^{-m}|v|_{H^m(\Omega)}.
\end{equation}

Then the following
approximation properties hold:
\begin{theorem}\label{theorem:3.1.3.1}
If $\alpha\in(0,\frac{1}{2})\cup (\frac{1}{2},1)$, and $u(x)\in I_{a+}^{\alpha}[H^m(\Omega)]$, then there exists a constant $C=C(\alpha,\Omega,m)$, such that
\begin{equation}\label{equation:lemma5}
\|u-Q_N^{\alpha} u\|_{L_2(\Omega)}\leq C N^{-m}\|\,_{a}D_{x}^{\alpha} u\|_{H^m(\Omega)},
\end{equation}
where $Q_N^\alpha u(x)$ denotes $\,_{a}I_{x}^{\alpha}\left(\Pi_N \,_{a}D_{x}^{\alpha}u\right)(x)$ or $\,_{a}I_{x}^{\alpha}\left(I_N \,_{a}D_{x}^{\alpha}u\right)(x)$.
\end{theorem}
\begin{proof}
If $u(x)\in I_{a+}^{\alpha}[H^m(\Omega)]$, then there is $v(x)=\,_{a}D_{x}^{\alpha}u(x)\in H^m(\Omega)$, such that $u(x)=\,_{a}I_{x}^{\alpha}v(x)$. Also, $u-Q_N^{\alpha} u\in I_{a+}^{\alpha}[L_2(\Omega)]$.

By (\expandafter{\romannumeral2}) and (\expandafter{\romannumeral3}) of Lemma \ref{lemma:2.2.4},
$$u-Q_N^{\alpha} u(x)= \,_{a}I_{x}^{\alpha}[v-Q_N v](x)\in L_q(\Omega),$$
where $q=\frac{2}{1-2\alpha}>2$ if $\alpha\in(0,\frac{1}{2})$; $q$ be any positive number if $\alpha\in(\frac{1}{2},1)$, and we take $q>2$. So,
$$\|u-Q_N^{\alpha} u\|_{L_2(\Omega)}
\leq C(\alpha)\|u-Q_N^{\alpha} u\|_{L_q(\Omega)}
=C(\alpha)\cdot\|\,_{a}I_{x}^{\alpha}[v-Q_N v]\|_{L_q(\Omega)}
\leq C(\alpha,\Omega)\cdot\|v-Q_N v\|_{L_2(\Omega)}.$$
Then by Eqs. (\ref{equation:3.1.3.1}) and (\ref{equation:3.1.3.2}), there exists a constant $C=C(\alpha,\Omega,m)$, such that
$$
\|u-Q_N^{\alpha} u\|_{L_2(\Omega)}
\leq CN^{-m}|v|_{H^m(\Omega)}
=CN^{-m}|\,_{a}D_{x}^{\alpha}u(x)|_{H^m(\Omega)}
\leq CN^{-m}\|\,_{a}D_{x}^{\alpha}u(x)\|_{H^m(\Omega)}.
$$
\end{proof}

\begin{theorem}\label{theorem:3.1.3.1}
Let $\alpha\in(0,\frac{1}{2})$, $p_2=\frac{2}{1+2\alpha}$. If $u(x)\in I_{a+}^{\alpha}[W^{m,p_2}(\Omega)]\cap I_{b-}^{\alpha}[W^{m,p_2}(\Omega)]$, then there exists a constant $C_1=C_1(\alpha,\Omega,m)$, such that
\begin{equation}\label{equation:lemma5}
\|u-Q_N u\|_{L_2(\Omega)}
\leq C_1 N^{-m}\|\,_{a}D_{x}^{\alpha} u\|_{W^{m,p_2}(\Omega)}
(\textrm{~or~}\leq C_1 N^{-m}\|\,_{x}D_{b}^{\alpha} u\|_{W^{m,p_2}(\Omega)}),
\end{equation}
where $Q_N u(x)$ denotes $\Pi_N u(x)$ or $I_N u(x)$.

Let $\alpha\in(\frac{1}{2},1)$, $p_2'=\frac{2}{2\alpha-1}$. If $u(x)\in I_{a+}^{\alpha}[W^{m,p_2'}(\Omega)]\cap I_{b-}^{\alpha}[W^{m,p_2'}(\Omega)]$, then there exists a constant $C_2=C_2(\alpha,\Omega,m)$, such that
\begin{equation}\label{equation:lemma5}
\|u-Q_N u\|_{L_2(\Omega)}
\leq C_2 N^{-m-1}\|\,_{a}D_{x}^{\alpha} u\|_{W^{m,p_2'}(\Omega)}
(\textrm{~or~}\leq C_2 N^{-m-1}\|\,_{x}D_{b}^{\alpha} u\|_{W^{m,p_2'}(\Omega)}).
\end{equation}
\end{theorem}
\begin{proof}
We only show the proof for the case when $\alpha\in(0,\frac{1}{2})$.
By Theorem \ref{theorem:3.1.2.1}, $u\in H^{m}(\Omega)$. Then
by Eqs. (\ref{equation:3.1.3.1}) and (\ref{equation:3.1.3.2}),
$$\|u-Q_Nu\|_{L_2(\Omega)}\leq CN^{-m}|D^m u|_{L_2(\Omega)}.$$
Together with Eqs. (\ref{equation:3.1.2.1}) and (\ref{equation:3.1.2.2}), we have
$$\|u-Q_Nu\|_{L_2(\Omega)}
\leq CN^{-m}|\,_{a}D_{x}^{\alpha} u|_{W^{m,p_2}(\Omega)}
\leq C_1N^{-m}\|\,_{a}D_{x}^{\alpha} u\|_{W^{m,p_2}(\Omega)}.$$
\end{proof}

Here we demonstrate two simple examples to numerically show the $L_2$ errors of $|u(x)-\,_{a}I_{x}^{\alpha}\left(\Pi_N \,_{a}D_{x}^{\alpha}u\right)(x)|$, where $\Omega=[-1,1]$, $u(x)\in I_{(-1)+}^{\alpha}[H^m(\Omega)]$, with $u(x)=1\in I_{(-1)+}^{\alpha}[L_2(\Omega)]$ for $0<\alpha<\frac{1}{2}$ by Remark \ref{remark:3}, and $u(x)=u_1(x)+u_2(x)\in I_{(-1)+}^{\alpha}[H^3(\Omega)]$, $u_1(x)=(1+x)^\alpha x^2$,
  $u_2(x)=\left\{
    \begin{array}{ll}
      -\frac{(x+1)^{3}}{6}   &~ [-1,0], \\
      \frac{2x^{3}-(x+1)^{3}}{6}   &~(0,1],
    \end{array}\right.$
respectively.

\begin{remark}
The reason why we choose these two specific test functions is that their convergence rates demonstrated in Figure \ref{fig1} just correspond to those in their classical cases, i.e., let $\alpha=0$. As a matter of fact, for classical situation, readers can test any jump function or function whose third derivative is a jump function, say $v(x)$, to check that the convergence rate of $|v(x)-Q_N v(x)|$ is similar as that in Figure \ref{fig1}, but not contradict to the classical results in (\ref{equation:3.1.3.1}) and (\ref{equation:3.1.3.2}).
\end{remark}

\begin{figure}[h!]
\centering
\subfloat[$u(x)=1$]{\includegraphics[scale=0.19]{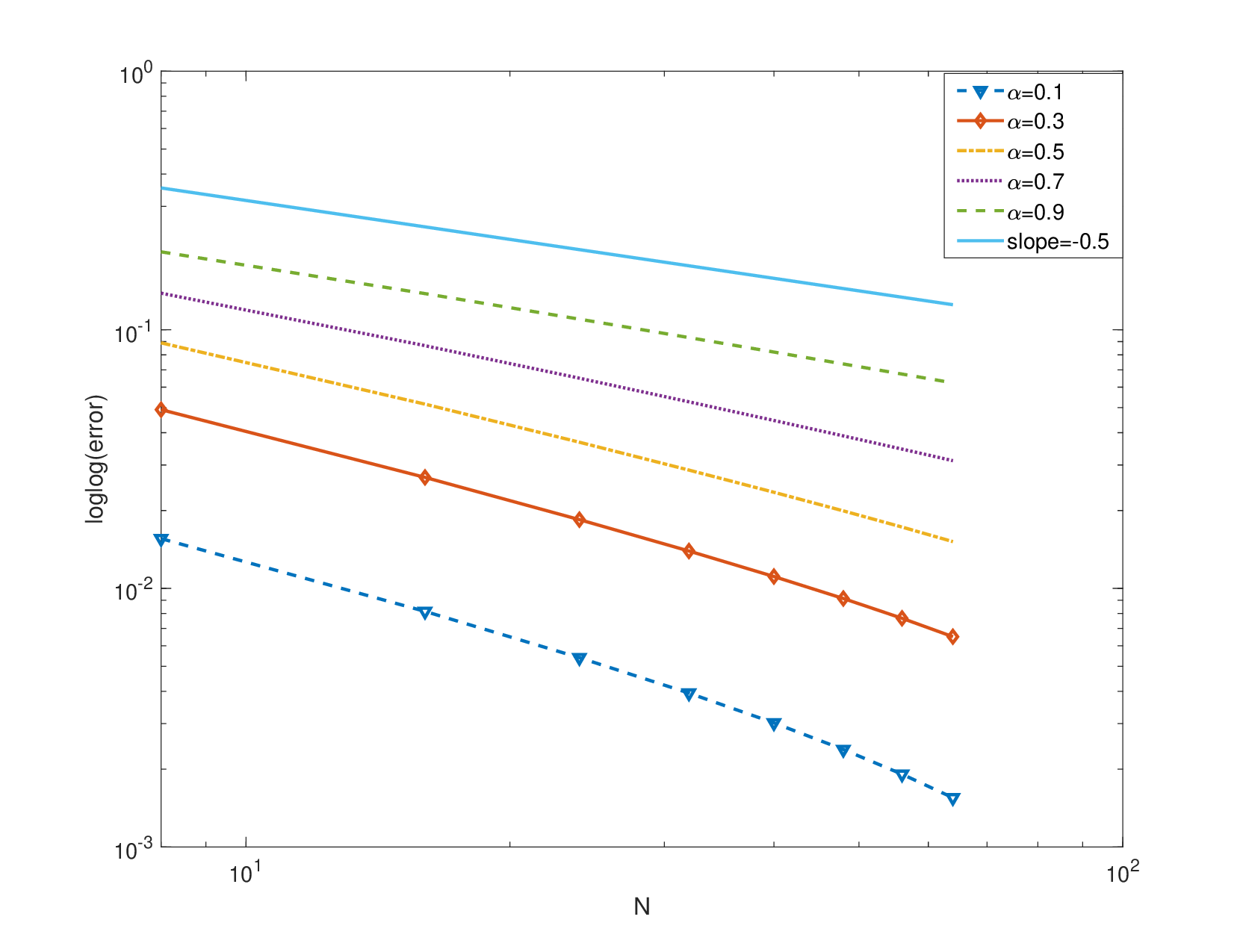}}
\subfloat[$u(x)=u_1(x)+u_2(x)$.]{\includegraphics[scale=0.3]{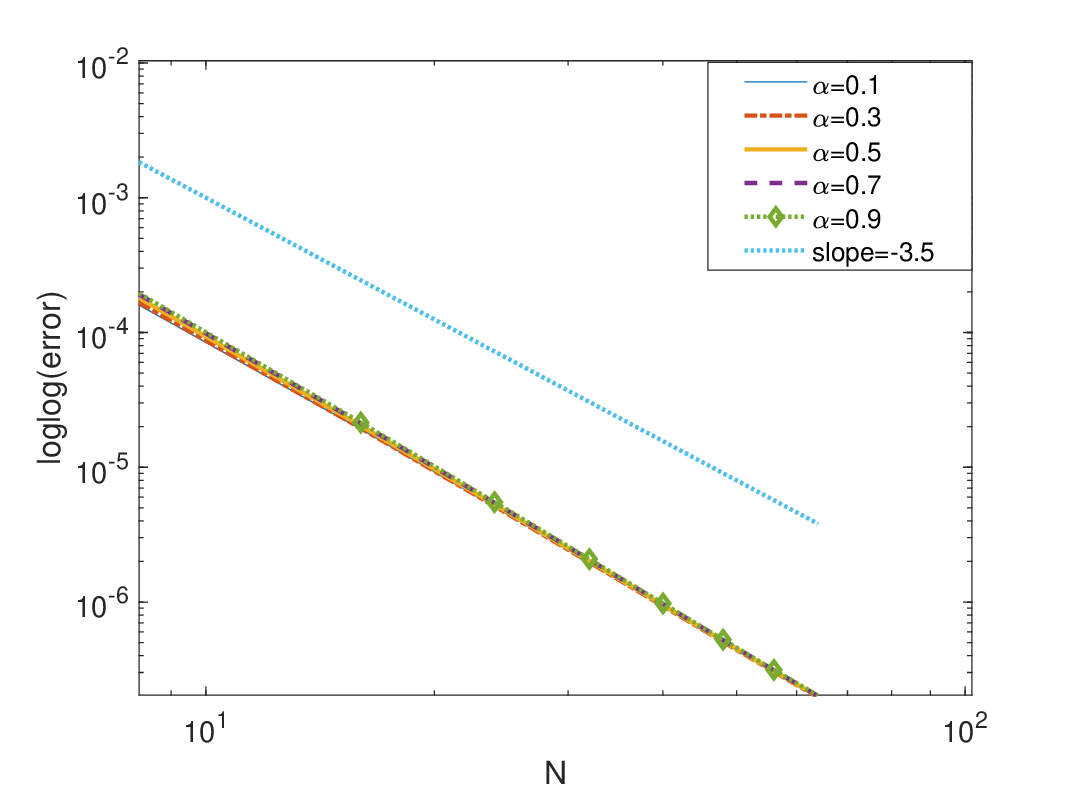}}
\caption{The numerical $L_2$ errors $|u(x)-\,_{a}I_{x}^{\alpha}\left(\Pi_N \,_{a}D_{x}^{\alpha}u\right)(x)|$  vs the polynomial degree N for different $\alpha\in(0,1)$.}
\label{fig1}
\end{figure}

\section{Application on Solving Fractional Equations}\label{section:4}

In this section, we try to offer a starting point on how to make use of the image space $I_{a+}^{\alpha}[W^{m,p}(\Omega)]$ or $I_{b-}^{\alpha}[W^{m,p}(\Omega)]$ when numerically solving fractional equations. As a matter of fact, Theorem \ref{theorem:3.1.3} or Theorems \ref{theorem:3.1.1} and \ref{theorem:3.1.2} can be taken advantage in a very flexible way when solving fractional equations just by choosing suitable parameter $\alpha$. From a mathematical view, by the discussions in above sections, two-sided fractional problems are actually even easier to handle during the approximation comparing with the one-sides problems, just by ignoring the term $(x-a)^\alpha p_{m-1}(x)$ or $(b-x)^\alpha p_{m-1}(x)$. 
So here we only provide some instructions on how to solve the following typical problems:
\begin{equation}\label{equation4.0.0}
\left\{ \begin{array}{lll}
{}_{a}D_{x}^{\gamma_1}u(x)+ d\cdot{}_{a}D_{x}^{\gamma_2}u(x)=h(x),\\
\textrm{given~conditions~on~}\partial\Omega,
\end{array} \right.
\end{equation}
with $0<\gamma_2<\gamma_1<2$. Also we set aside on discussing the physical meaning of the boundary conditions.

It is clear that Problem (\ref{equation4.0.0}) can be analytically solved in theory, and thus the regularity of $u(x)$ is clear, if $\gamma_1$, $\gamma_2$ and $h(x)$ are fixed. Therefore, in the following part, for different classifications of $\gamma_1$, $\gamma_2$ and $h(x)$, we discuss how to choose suitable basis functions when approximating the solution. Without loss of generality, we suppose $\Omega=[-1,1]$.

Here we only discuss four special cases.

\textbf{Case 1. $u(x)\in H^m(\Omega)$}:

By Theorem \ref{theorem:3.1.3} or Theorem \ref{theorem:3.1.1}, if $u(x)\in H^m(\Omega)$, $m\geq 2$, then ${}_{-1}D_{x}^{\gamma}u(x)=D^2{}_{-1}I_{x}^{2-\gamma}u(x)=\tilde{h}_1(x,\gamma)+\tilde{h}_2(x,\gamma)$, where
$$\tilde{h}_1(x,\gamma)\in(x+1)^{-\gamma}P_{m-1}(x),$$
and
$$\tilde{h}_2(x,\gamma)=o((x+1)^{m-\gamma-\frac{1}{2}})~\textrm{as~} x\rightarrow -1,$$
with
$$\tilde{h}_2(x,\gamma)\in
\left\{
\begin{array}{lll}
W^{m-1,q_1'}(\Omega),& q_1'=\frac{2}{2\gamma-1},&
\gamma\in(1,\frac{3}{2}),\vspace{5pt}\\
W^{m-2,q_1}(\Omega),& q_1=\frac{2}{2\gamma-3},&
\gamma\in(\frac{3}{2},2).
\end{array}
\right.
$$
Since $\tilde{h}_2(x,\gamma_1)$ implies $\tilde{h}_2(x,\gamma_2)$ when $\gamma_1>\gamma_2$, so if $u(x)\in H^m(\Omega)$, then $h(x)$ in Problem (\ref{equation4.0.0}) can be rewritten as \begin{equation}\label{equation:4.0.1_1}
h(x)=h_1(x,\gamma_1)+h_1(x,\gamma_2)+h_2(x,\gamma_1),
\end{equation}
where
\begin{subequations}
\begin{equation}\label{equation:4.0.1}
h_1(x,\gamma_1)\in(x+1)^{-\gamma_1}P_{m-1}(x),
\end{equation}
\begin{equation}\label{equation:4.0.2}
h_1(x,\gamma_2)\in(x+1)^{-\gamma_2}P_{m-1}(x),
\end{equation}
and
\begin{equation}\label{equation:4.0.3}
h_2(x,\gamma_1)=o((x+1)^{m-\gamma_1-\frac{1}{2}})~\textrm{as~} x\rightarrow -1,
\end{equation}
with
\begin{equation}\label{equation:4.0.4}
h_2(x,\gamma_1)\in
\left\{
\begin{array}{lll}
W^{m-1,q_1'}(\Omega_1),& q_1'=\frac{2}{2\gamma_1-1},&
\gamma_1\in(1,\frac{3}{2}),\vspace{5pt}\\
W^{m-2,q_1}(\Omega),& q_1=\frac{2}{2\gamma_1-3},&
\gamma_1\in(\frac{3}{2},2).
\end{array}
\right.
\end{equation}
\end{subequations}
On the other hand, we can also in fact prove similarly to the analysis in above sections that: if $\gamma_1$, $\gamma_2$, and $h(x)$ in Problem (\ref{equation4.0.0}) satisfies Eqs. (\ref{equation:4.0.1_1})-(\ref{equation:4.0.4}), then $u(x)\in H^m(\Omega)$.

\textbf{So, if these are the cases for $\gamma_1$, $\gamma_2$, and $h(x)$ in Problem (\ref{equation4.0.0}), then we let
$$ u(x)\approx u_N(x):=\sum_{n=0}^N u_n\cdot L_n(x), \textrm{~and~$u_N(x)$~satisfies~given~conditions~on~}\partial\Omega,$$
where $\{u_n\}_{n=0}^N$ is a set of coefficients to be determined, and $\{L_n(x)\}_{n=0}^{\infty}$ is a set of Legendre polynomials which are orthogonal basis functions in $L_2([-1,1])$ space~\cite{Hesthaven:07}.}
%

\textbf{Case 2. $u(x)\in I_{(-1)+}^{\beta}[H^m(\Omega)]$, $0<\beta<1$}:

Given $0<\beta<1$. By Theorem \ref{theorem:3.1.1} and Theorem \ref{theorem:3.1.3}, if $u(x)\in I_{(-1)+}^{\beta}[H^m(\Omega)]$, $m\geq 2$, then ${}_{-1}D_{x}^{\gamma}u(x)=D^2{}_{-1}I_{x}^{2-\gamma}u(x)=\hat{h}_1(x,\gamma,\beta)+\hat{h}_2(x,\gamma,\beta)$, where
$$\hat{h}_1(x,\gamma,\beta)\in(x+1)^{-\gamma+\beta}P_{m-1}(x),$$
and
$$\hat{h}_2(x,\gamma,\beta)=o((x+1)^{m-\gamma+\beta-\frac{1}{2}})~\textrm{as~} x\rightarrow -1,$$
with
$$\hat{h}_2(x,\gamma,\beta)\in
\left\{
\begin{array}{lll}
W^{m-2,p_1}(\Omega),& \frac{1}{p_1}=\gamma-\beta-\frac{3}{2},&
\gamma>\beta+\frac{3}{2},\vspace{5pt}\\
W^{m-1,p_1'}(\Omega),& \frac{1}{p_1'}=\gamma-\beta-\frac{1}{2},&
\gamma<\beta+\frac{3}{2},
\end{array}
\right.
$$
when $\beta\in(0,\frac{1}{2})$; and
$$\hat{h}_2(x,\gamma,\beta)\in
\left\{
\begin{array}{lll}
W^{m-2,p_1}(\Omega),& \frac{1}{p_1}=\gamma-\beta-\frac{1}{2},&
\gamma>\beta+\frac{1}{2},\vspace{5pt}\\
W^{m-1,p_1'}(\Omega),& \frac{1}{p_1'}=\gamma-\beta+\frac{1}{2},&
\gamma<\beta+\frac{1}{2}.
\end{array}
\right.
$$
when $\beta\in(\frac{1}{2},1)$.

Since $\hat{h}_2(x,\gamma_1)$ implies $\hat{h}_2(x,\gamma_2)$ when $\gamma_1>\gamma_2$, so if $u(x)\in I_{(-1)+}^{\beta}[H^m(\Omega)]$, then $h(x)$ in Problem (\ref{equation4.0.0}) can be rewritten as \begin{equation}\label{equation:4.1.1_1}
h(x)=\bar{h}_1(x,\gamma_1,\beta)+\bar{h}_1(x,\gamma_2,\beta)+\bar{h}_2(x,\gamma_1,\beta),
\end{equation}
where
\begin{subequations}
\begin{equation}\label{equation:4.1.1}
\bar{h}_1(x,\gamma_1,\beta)\in(x+1)^{-\gamma_1+\beta}P_{m-1}(x),
\end{equation}
\begin{equation}\label{equation:4.1.2}
\bar{h}_1(x,\gamma_2,\beta)\in(x+1)^{-\gamma_2+\beta}P_{m-1}(x),
\end{equation}
and
\begin{equation}\label{equation:4.1.3}
\bar{h}_2(x,\gamma_1,\beta)=o((x+1)^{m-\gamma_1+\beta-\frac{1}{2}})~\textrm{as~} x\rightarrow -1,
\end{equation}
with
\begin{equation}\label{equation:4.1.4}
\bar{h}_2(x,\gamma_1,\beta)\in
\left\{
\begin{array}{lll}
W^{m-2,p_1}(\Omega),& \frac{1}{p_1}=\gamma_1-\beta-\frac{3}{2},&
\gamma_1>\beta+\frac{3}{2},\vspace{5pt}\\
W^{m-1,p_1'}(\Omega),& \frac{1}{p_1'}=\gamma_1-\beta-\frac{1}{2},&
\gamma_1<\beta+\frac{3}{2},
\end{array}
\right.
\end{equation}
when $\beta\in(0,\frac{1}{2})$; and
\begin{equation}\label{equation:4.1.5}
\bar{h}_2(x,\gamma_1,\beta)\in
\left\{
\begin{array}{lll}
W^{m-2,p_1}(\Omega),& \frac{1}{p_1}=\gamma_1-\beta-\frac{1}{2},&
\gamma_1>\beta+\frac{1}{2},\vspace{5pt}\\
W^{m-1,p_1'}(\Omega),& \frac{1}{p_1'}=\gamma_1-\beta+\frac{1}{2},&
\gamma_1<\beta+\frac{1}{2}.
\end{array}
\right.
\end{equation}
when $\beta\in(\frac{1}{2},1)$.
\end{subequations}

\textbf{So if $\gamma_1$, $\gamma_2$, and $h(x)$ in Problem (\ref{equation4.0.0}) satisfies Eqs. (\ref{equation:4.1.1_1})-(\ref{equation:4.1.5}), then we let
$$u(x)\approx \bar{u}_N(x):=\sum_{n=0}^N\bar{u}_n\cdot {}_{-1}I_{x}^{\beta}L_n(x)
=(1+x)^{\beta}\sum_{n=0}^N \frac{\Gamma(n+1)}{\Gamma(n+1+\beta)}\cdot\bar{u}_n J_n^{-\beta,\beta}(x),$$
and $\bar{u}_N(x)$ satisfy given conditions on $\partial\Omega$ at the same time, where $\{\bar{u}_n\}_{n=0}^N$ is a set of coefficients to be determined, and $\{J_{n}^{-\beta,\beta}(x)\}_{n=0}^N$ are Jacobi polynomials~\cite{Askey:75}.}
\begin{remark}
In this case, $\beta$ is often chosen as $\beta=\frac{\gamma_1}{2}$ in the variational weak formulation.
\end{remark}

\textbf{Case 3. $u(x)\in I_{(-1)+}^{\gamma_1}[H^{m}(\Omega)]$}:

If $u(x)\in I_{(-1)+}^{\gamma_1}[H^{m}(\Omega)]$, $m\geq 1$,
then
$h(x)\in  H^{m}(\Omega) \cap I_{(-1)+}^{\gamma_1-\gamma_2}[H^{m}(\Omega)]$. Since by Theorem \ref{theorem:3.1.1} or Theorem \ref{theorem:3.1.3}, $v(x)\in I_{(-1)+}^{\gamma_1-\gamma_2}[H^{m}(\Omega)]$ if and only if $v(x)=v_1(x)+v_2(x)$, where
$$v_1(x)\in(x+1)^{\gamma_1-\gamma_2}P_{m-1}(x),$$
and
$$v_2(x)=o((x+1)^{m+\gamma_1-\gamma_2-\frac{1}{2}})~\textrm{as~} x\rightarrow -1,$$
with
$$v_2(x)\in
\left\{
\begin{array}{lll}
W^{m,p_1}(\Omega),& p_1=\frac{2}{1-2(\gamma_1-\gamma_2)},&
\gamma_1-\gamma_2\in(0,\frac{1}{2}),\vspace{5pt}\\
W^{m+1,p_1'}(\Omega),& p_1'=\frac{2}{3-2(\gamma_1-\gamma_2)},&
\gamma_1-\gamma_2\in(\frac{1}{2},1),\vspace{5pt}\\
W^{m+1,q_1}(\Omega),& q_1=\frac{2}{3-2(\gamma_1-\gamma_2)},&
\gamma_1-\gamma_2\in(1,\frac{3}{2}),\vspace{5pt}\\
W^{m+2,q_1'}(\Omega),& q_1'=\frac{2}{5-2(\gamma_1-\gamma_2)},&
\gamma_1-\gamma_2\in(\frac{3}{2},2).
\end{array}
\right.$$
Of course, $v_2(x)\in H^m(\Omega)$.

\textbf{So, if}
\begin{equation}\label{equation:4.3.1_1}
h(x)=\check{h}_1(x,\gamma_1,\gamma_2)+\check{h}_2(x,\gamma_1,\gamma_2),
\end{equation}
\textbf{where}
\begin{subequations}
\begin{equation}\label{equation:4.3.1}
\check{h}_1(x,\gamma_1,\gamma_2)\in(x+1)^{\gamma_1-\gamma_2}P_{m-1}(x),
\end{equation}
\textbf{and}
\begin{equation}\label{equation:4.3.2}
\check{h}_2(x,\gamma_1,\gamma_2)\in H^m(\Omega),~~~
\check{h}_2(x,\gamma_1,\gamma_2)=o((x+1)^{m+\gamma_1-\gamma_2-\frac{1}{2}})~\textrm{as~} x\rightarrow -1,
\end{equation}
\end{subequations}
\textbf{then we let
$$u(x)\approx \check{u}_N(x):=\sum_{n=0}^N\check{u}_n\cdot {}_{-1}I_{x}^{\gamma_1}L_n(x)
=(1+x)^{\gamma_1}\sum_{n=0}^N \frac{\Gamma(n+1)}{\Gamma(n+1+\gamma_1)}\cdot\check{u}_n J_n^{-\gamma_1,\gamma_1}(x),$$
and $\bar{u}_N(x)$ satisfies given conditions on $\partial\Omega$ at the same time, where $\{\check{u}_n\}_{n=0}^N$ is a set of coefficients to be determined.}

\textbf{Case 4. $u(x)\in W^{m,p}(\Omega)$, with $p=\frac{2}{3-2\gamma_1}$ when $\gamma_1\in(1,\frac{3}{2})$,
$p=\frac{2}{5-2\gamma_1}$ when $\gamma_1\in(\frac{3}{2},2)$}:

If $u(x)\in W^{m,p}(\Omega)$, with $p=\frac{2}{3-2\gamma_1}$ when $\gamma\in(1,\frac{3}{2})$,
$p=\frac{2}{5-2\gamma}$ when $\gamma\in(\frac{3}{2},2)$, then by Theorem \ref{theorem:3.1.2} or Theorem \ref{theorem:3.1.3},
${}_{-1}D_{x}^{\gamma}u(x)=D^2{}_{-1}I_{x}^{2-\gamma}u(x)=\acute{h}_1(x,\gamma)+\acute{h}_2(x,\gamma)$, where
$$\acute{h}_1(x,\gamma)\in(x+1)^{-\gamma}P_{m-1}(x),$$
and $\acute{h}_2(x,\gamma)\in H^{m-1}$, with
$$\acute{h}_2(x,\gamma)=o((x+1)^{m-\frac{3}{2}})~\textrm{as~} x\rightarrow -1,$$
when $\gamma\in(1,\frac{3}{2})$; and $\acute{h}_2(x,\gamma)\in H^{m-2}$ with
$$\acute{h}_2(x,\gamma)=o((x+1)^{m-\frac{5}{2}})~\textrm{as~} x\rightarrow -1,$$
when $\gamma\in(\frac{3}{2},2)$.

Since $\acute{h}_2(x,\gamma_1)$ implies $\acute{h}_2(x,\gamma_2)$ when $\gamma_1>\gamma_2$, so if $u(x)\in W^{m,p}(\Omega)$, then $h(x)$ in Problem (\ref{equation4.0.0}) can be rewritten as \begin{equation}\label{equation:4.4.1_1}
h(x)=\grave{h}_1(x,\gamma_1)+\grave{h}_1(x,\gamma_2)+\grave{h}_2(x,\gamma_1),
\end{equation}
where
\begin{subequations}
\begin{equation}\label{equation:4.4.1}
\grave{h}_1(x,\gamma_1)\in(x+1)^{-\gamma_1}P_{m-1}(x),
\end{equation}
\begin{equation}\label{equation:4.4.2}
\grave{h}_1(x,\gamma_2)\in(x+1)^{-\gamma_2}P_{m-1}(x),
\end{equation}
and
\begin{equation}\label{equation:4.4.3}
\grave{h}_2(x,\gamma_1)\in
H^{m-1}(\Omega),~~\grave{h}_2(x,\gamma_1)=o((x+1)^{m-\frac{3}{2}})~\textrm{as~} x\rightarrow -1,
\end{equation}
when $\gamma_1\in(1,\frac{3}{2})$; and
\begin{equation}\label{equation:4.4.5}
\grave{h}_2(x,\gamma_1)\in
H^{m-2}(\Omega),~~\grave{h}_2(x,\gamma_1)=o((x+1)^{m-\frac{5}{2}})~\textrm{as~} x\rightarrow -1,
\end{equation}
when $\gamma_1\in(\frac{3}{2},2)$.
\end{subequations}

\textbf{So, if $\gamma_1$, $\gamma_2$, and $h(x)$ in Problem (\ref{equation4.0.0}) satisfy Eqs. (\ref{equation:4.4.1_1})-(\ref{equation:4.4.5}), then we let
$$u(x)\approx \tilde{u}_N(x):=p_{m-1}(x)+\sum_{n=0}^N\tilde{u}_n\cdot {}_{-1}D_{x}^{2-\gamma_1}\phi_n(x),$$
and $\tilde{u}_N(x)$ satisfies given conditions on $\partial\Omega$ at the same time, where $p_{m-1}(x)$ and $\{\tilde{u}_n\}_{n=0}^N$ are to be determined, and $\{\phi_n(x)\}_{n=0}^N$ are liner combinations of Legendre polynomials which satisfy $\phi_n^{k}(-1)=0,~k=0,1,\cdots,m,$ when $\gamma_1\in(1,\frac{3}{2})$; and $\phi_n^{k}(-1)=0,~k=0,1,\cdots,m-1,$ when $\gamma_1\in(\frac{3}{2},2)$.}

\begin{remark}
In fact, (\ref{equation4.0.0}) can be generalized to linear fractional differential equation with sequential or multi-term derivatives as discussed in~\cite{Podlubny:99}, which can be likewise be studied without additional difficulty.
\end{remark}
\section{Conclusion}\label{section:5}
Fractional derivative operators, which are widely used in abnormal and non-local models, are actually pseudo-differential. The key difficulty lies in the Riemann-Liouville integral part of the operators.
In this paper, we discuss the regularity properties of the fractional operators by using the image spaces of Riemann-Liouville integral operators on $W^{m,p}(\Omega)$. Many useful results can be obtained without tedious proof thanks to the special properties of the defined spaces. Functions in the image spaced introduce in this paper consist two parts: one part is smooth in the domain but possesses specific regularity at the boundary point; another part shows the regularity in the domain and the asymptotic behavior tending to zero near the boundary point---all of which can be represented by functions in classical Sobolev spaces. So the image spaces we discussed distinguish the fractional derivative spaces that are defined as the closures of $C_0^\infty(\Omega)$ functions under the norm of fractional Sobolev spaces on $\mathbb{R}$.
What is more, the tempered operators and the corresponding Riemann-Liouville ones show to be reciprocal in the image spaces of Riemann-Liouville integrals on $L^{p}(\Omega)$, which is expected to make some efforts on theoretical support for relevant numerical methods. We also offer a starting point on how to make use of the image space $I_{a+}^{\alpha}[W^{m,p}(\Omega)]$ or $I_{b-}^{\alpha}[W^{m,p}(\Omega)]$ when numerically solving fractional equations. In our future work, we shall further apply the spaces into numerical methods for solving different kinds of specific non-local problems.


\Acknowledgements{The authors thank Professor Xiao-Bing Feng for the discussions, and thank Professor Pingwen Zhang for the academic support. The first author was supported by National Natural Science Foundation of China (Grant No. 11801448), by the Natural Science Basic Research Plan in Shaanxi Province of China under Grant 2018JQ1022. The second author was supported by National Natural Science Foundation of China (Grant No. 11271173).}






\begin{thebibliography}{99}

\bibitem{Alikhanov:15}
Alikhanov A A. A new difference scheme for the time fractional diffusion equation. J Comput Phys, 2015, 280: 424-438

\bibitem{Askey:75}
Askey R.
Orthogonal Polynomials and Special Functions. SIAM: 1975

\bibitem{Bouchaud:90}
Bouchaud J, Georges A. Anomalous diffusion in disordered media:
Statistical mechanisms, models and physical applications. Phys Rep, 1990, 195: 127-293

\bibitem{Butzer:00}
Butzer P L, Westphal U. An Introduction to Fractional Calculus. World
Scientific, Singapore: 2000

\bibitem{Cartea:07}
Cartea \'{A}, del-Castillo-Negrete D. Fluid limit of the continuous-time random walk with general L\'{e}vy jump distribution functions. Phys Rev E, 2007, 76: 041105

\bibitem{Chen:15b}
Chen M H, Deng W H.
Discretized fractional substantial calculus.
ESAIM: Math Model Num(2), 2015, 49: 373-394

\bibitem{Chen:16}
Chen S, Shen J, and Wang L L. Generalized Jacobi functions and their applications to
fractional differential equations. Math Comp, 2016, 85: 1603-1638


\bibitem{Deng:07d}
Deng W H. Numerical algorithm for the time fractional Fokker-Planck equation. J Comput Phys, 2007, 227: 1510-1522

\bibitem{Deng:08}
Deng W H. Finite element method for the space and time fractional Fokker-Planck equation. SIAM J. Numer Anal, 2008, 47: 204-226

\bibitem{Deng:19}
Deng W H, Zhang Z J. High Accuracy Algorithm for the Differential Equations Governing Anomalous Diffusion. World Scientific, Singapore: 2019

\bibitem{Diethelm:10}
Diethelm K. The analysis of fractional differential equations. Springer-Verlag, Berlin: 2010

\bibitem{Ervin:06}
Ervin V J, Roop J P.
Variational formulation for the stationary fractional advection dispersion
equation. Numer Methods Partial Differential Equations, 2006, 22: 558-576

\bibitem{Ervin:16}
Ervin V J, Heuer N, Roop J P. Regularity of the solution to 1-D fractional order diffusion equations. Math.  Comput, 2018, 87:2273-2294.

\bibitem{Evans:10}
Evans L C.
Partial Differential Equations: Second Edition. American Mathematical Society, Providence: 2010

\bibitem{Gakhov:66}
Gakhov F D. Boundary Value Problems. Pergamon: 1966

\bibitem{Hesthaven:07}
Hesthaven J S, Gottlieb S, Gottlieb D.
Spectral Methods for Time-Dependent Problems. Cambridge University Press: 2007

\bibitem{Jiao:16}
Jiao Y, Wang L L, Huang C. Well-conditioned fractional collocation methods using fractional birkhoff interpolation basis. J Comput Phys, 2016, 305: 1-28

\bibitem{Jin:17}
Jin B T, Li B Y, Zhou Z. Correction of high-order bdf convolution quadrature for fractional evolution equations. SIAM J Sci Comput(6), 2017, 39: A3129-A3152


\bibitem{Kenneth:93}
Kenneth S M, Bertram R. An Introduction to the Fractional Calculus and
Fractional Differential Equations. Wiley, New York: 1993

\bibitem{Kolmogorov:68}
Kolmogorov A N, Fomin S V.
Fundamentals of the Theory of Functions
and Functional Analysis. Nauka, Moscow: 1968


\bibitem{Kyprianou:18}
Kyprianou A E, Osojnik A, Shardlow T. Unbiased ``walk-on-spheres" Monte Carlo methods for
the fractional Laplacian. IMA Journal of Numerical Analysis(3). 2018, 38: 1550-1578

\bibitem{Li:09}
Li X J, Xu C J.
A space-time spectral method for the time fractional diffusion
equation. SIAM J Numer Anal, 2009, 47: 2108-2131

\bibitem{Magdziarz:07}
Magdziarz M, Weron A. Competition between subdiffusion and L\'{e}vy flights: A Monte Carlo approach. Phys Rev E, 2007, 75: 056702

\bibitem{Mao:18}
Mao Z P, Karniadakis G. A spectral method (of exponential convergence) for singular solutions of the
diffusion equation with general two-sided fractional derivative. SIAM J Numer Anal(1), 2018, 56: 24-49


\bibitem{Meerschaert:04}
Meerschaert M M, Tadjeran C. Finite difference approximations for fractional advection-dispersion
flow equations. J Comput Appl Math, 2004, 172: 65-77

\bibitem{Metzler:00}
Metzler R, Klafter J. The random walk's guide to
anomalous diffusion:
A fractional dynamics approach. Phys Rep, 2000, 339: 1-77

\bibitem{Miller:93}
Miller K S, Ross B. An Introduction to the Fractional Calculus and Fractional Differential Equations. Wiley, New York: 1993

\bibitem{Oldham:74}
Oldham K B, Spanier J. The Fractional Calculus. Academic Press, New York: 1974

\bibitem{Podlubny:99}
Podlubny I. Fractional Differential Equations. Academic Press, New York: 1999

\bibitem{Sabzikar:15}
Sabzikar F, Meerschaert M M, Chen J H. Tempered fractional calculus. J Comput Phys, 2015, 293: 14-28

\bibitem{Samko:93}
Samko S, Kilbas A, Marichev O.
Fractional Integrals and Derivatives: Theory and Applications. Gordon and Breach, Amsterdam: 1993

\bibitem{Tian:14}
Tian W Y, Deng W H, Wu Y J.
Polynomial spectral collocation method for space fractional
advection-diffusion equation. Numer Methods Partial Differential Equations(2), 2014, 30: 514-535

\bibitem{Tian:12}
Tian W Y, Zhou H, Deng W H. A class of second order difference approximations
for solving space fractional diffusion Equations. Math Comp, 2015, 84: 1703-1727

\bibitem{Wang:15}
Wang H, Zhang X. A high-accuracy preserving spectral Galerkin method for the Dirichlet
boundary-value problem of variable-coefficient conservative fractional diffusion equations. J
Comput Phys, 2015, 281: 67-81

\bibitem{Zayernouri:15}
Zayernouri M, Ainsworth M, Karniadakis G E. A unified Petrov-Galerkin spectral
method for fractional PDEs. Comput Methods Appl Mech Engrg, 2015, 283: 1545-1569

\bibitem{Zeng:18}
Zeng F H, Turner I, Burrage K. A stable fast time-stepping method for fractional integral and derivative operators. J Sci Comput., 2018, 77: 283-307

\bibitem{Zhao:14}
Zhao L J, Deng W H. Jacobian-predictor-corrector approach for fractional differential equations. Adv Comput Math, 2014, 40: 137-165

\bibitem{Zhao:16}
Zhao L J, Deng W H. High order finite difference methods on non-uniform
meshes for space fractional operators. Adv Comput Math, 2016, 42: 425-468



\end{thebibliography}
\end{document}